\documentclass{article}

\usepackage{comment}
\usepackage[utf8]{inputenc}
\usepackage[T1]{fontenc}
\usepackage{amssymb}
\usepackage{amsmath, amsthm}
\usepackage{amsfonts}
\usepackage{graphics}
\usepackage{color}
\usepackage{xcolor}
\usepackage{graphicx}
\usepackage{graphics}
\usepackage{color}
\usepackage{amsmath}
\usepackage{amsfonts}
\usepackage{amssymb}
\usepackage{enumerate}
\usepackage{hyperref}
\usepackage[english]{babel}
\usepackage{enumitem}
\usepackage{bm}

\newtheorem{thm}{Theorem}

\newtheorem{cl}{Claim}
\newtheorem{Lemma}{Lemma}

 \title{Regular packing of rooted hyperforests with root constraints in hypergraphs}
\author{Pierre Hoppenot\footnote{ 
Univ. Grenoble Alpes, CNRS, Grenoble INP, G-SCOP, France},
\ \ Mathis Martin,
\footnote{Saint-Etienne, EMSE, France}
\ \ Zolt\'an Szigeti$^*$
}
\begin{document}
\maketitle

\begin{abstract}
The seminal papers of Edmonds \cite{Egy}, Nash-Williams \cite{NW} and Tutte \cite{Tu} have laid the foundations of  the theories of packing arborescences and packing  trees. The directed version has been extensively investigated, resulting in a great number of generalizations.  In contrast, the undirected version has been marginally considered. The aim of this paper is to further develop the theories of packing trees and forests.
Our main result on  graphs characterizes the existence of a packing of $k$ forests, $F_1, \ldots, F_k$, in a graph $G$ such that each vertex of $G$ belongs to exactly $h$ of the forests, and in addition, each $F_i$ has between $\ell(i)$ and $\ell'(i)$ connected components and the total number of connected components in the packing is  between $\alpha$ and  $\beta$. Finally, we extend this result to hypergraphs and dypergraphs, the latter giving a generalization of a theorem of B\'erczi and Frank~\cite{BF3}.
\end{abstract}

\section{Introduction}

While the theory of packing trees, started in 1961 with a result  of Nash-Williams~\cite{NW} and Tutte~\cite{Tu}, the theory of packing arborescences  started in the seventies with the results of Edmonds~\cite{Egy}, Lov\'asz~\cite{Lov}, and Frank~\cite{FA78}. After a long silence, a new wave of results generalizing these seminal results on packing arborescences and branchings appeared in the aughts due to Frank, Kir\'aly and   Kir\'aly~\cite{fkiki}, Kamiyama, Katoh and  Takizawa~\cite{japan}, Frank~\cite{F2}, and B\'erczi and  Frank~\cite{BF2, BF}. The last ten years the development of the theory accelerated, great number of new results appeared: \cite{BF3, Sz, FKLSzT, FKST, gy, gy2, HSz4, HSz5, cskir, KSzT, MSz}.
However, only a few results concerning packing trees and forests in undirected graphs have been published since 1961. We can cite Peng, Chen and  Koh~\cite{PCK}, Frank, Kir\'aly and  Kriesell~\cite{fkk}, and Katoh and  Tanigawa~\cite{KT}. The aim of this paper is to further develop the theory of packing forests.

Nash-Williams~\cite{NW} and Tutte~\cite{Tu} characterized graphs having a packing of $k$ spanning trees. In this paper, we concentrate on packing forests with some constraints on the number of their connected components. We will first prove a result on packing $k$ spanning forests with $\ell(1), \dots, \ell(k)$ connected components in a graph. We then introduce the concept of $h$-regular packing of forests, 
that is, each vertex belongs to $h$ of the forests. Note that a packing of $k$ spanning forests is equivalent to a $k$-regular packing of $k$ forests.
We generalize the above-mentioned result to $h$-regular packing of $k$ forests with $\ell(1),\dots,\ell(k)$ connected components in a graph, using some new technics developed in~\cite{hopp}. This will then lead to a further generalization to $h$-regular packing of $k$  forests satisfying bounds on the number of their connected components  and bounds on the total number of their connected components   in a graph.
The main result of the present paper, on $h$-regular packing of $k$ rooted hyperforests satisfying bounds on the number of their roots and bounds on the total number of their roots  in a hypergraph, is obtained from the graphic version by the technic of trimming.

We also study the directed counterpart of the previous theorem, that is, a generalization to $h$-regular packing in dypergraphs of a result of B\'erczi and  Frank \cite{BF3} that characterizes digraphs having a packing of $k$ spanning branchings satisfying bounds on the number of their roots and bounds on the total number of their roots. We prove it by applying an abstract theorem given in~\cite{BF3} along with a result of Fortier et al.~\cite{FKLSzT}.

Let us now justify some of the choices we made. For a forest, the number of its vertices  is equal to the number of its edges plus the number of its connected components. In the case of spanning forests, the number of vertices is always equal to the number of vertices of the graph, hence the problems of packing spanning forests with given numbers of connected components and with given numbers of edges are equivalent. For regular packings, the forests are not necessarily spanning, therefore this equivalence does not hold anymore. Since, according to Theorem~\ref{hkbrell}, deciding whether there exists  an $h$-regular packing of $k$ forests in a graph  containing $\ell(1),\dots,\ell(k)$ edges is NP-complete, we are obliged to work with the number of connected components.  However, in a hypergraph, the number of connected components of a hyperforest cannot be defined. This motivates the employment of  rooted hyperforests. In the case of graphs, since the number of roots of a rooted forest is the number of its connected components and rooted forests are less  used than forests, we will mainly speak about forests and connected components.

In  a digraph, since Theorem~\ref{hkbrell} also shows that deciding whether there exists  an $h$-regular packing of $k$ branchings   containing $\ell(1),\dots,\ell(k)$ arcs is NP-complete, we work with the sizes of the root sets.

The organization of the paper is as follows. In Section \ref{def} we give all the definitions needed.
In Section~\ref{preresults} we provide some preliminary results to be applied later in the proofs.
Section~\ref{knownresults} contains the old results.
We present in Section~\ref{newresults} our new results.
In Section~\ref{proofs} we provide the proofs of the new results.

\section{Definitions}\label{def}

In this section we provide all the definitions needed in the paper. For the basic definitions, see \cite{book}.

\subsection{General definitions}

The sets of integers, non-negative integers, and positive integers    are respectively denoted by $\pmb{\mathbb{Z}},\pmb{\mathbb{Z}_{\ge 0}}$, and $\pmb{\mathbb{Z}_{>0}}.$ 
For $k\in\mathbb Z_{>0},$ $\pmb{\mathbb{Z}_k}$ denotes the set $\{1,\dots,k\}.$ 
For a set $Q$, a subset $X$ of $Q$ and a function $m : Q \rightarrow \mathbb{Z},$ we define $\bm{m(X)} = \sum_{x \in X} m(x)$.
For a function $\ell: \mathbb{Z}\rightarrow \mathbb Z$ and $p\in\mathbb Z,$ we introduce  the function {\boldmath$\ell_p$} (that will be extensively employed throughout the paper) as 
\begin{equation*}
\ell_p(i)=\min\{\ell(i),p\} \text{ for } i\in \mathbb{Z}.
\end{equation*}

For a family $\mathcal{S}$ of subsets of $V$ and a subset $X \subseteq V$, we denote by $\bm{\mathcal{S}_X}=\{S \in \mathcal{S}: X \cap S \neq \emptyset\}$ the family of members of $\mathcal{S}$ that intersect $X$. 
For a subset $X \subseteq V,$ we denote $\bm{\overline{X}} = V - X$ (we use this notation only when it is clear what $V$ is in the context). 

A set function $p$ on $V$ is called \textit{supermodular} if  the following inequality holds
for all $X, Y \subseteq V$,
\begin{equation}
    p(X) + p(Y) \le p(X \cap Y) + p(X \cup Y).\label{supmod}
\end{equation}
We say that $p$ is \textit{intersecting supermodular} if \eqref{supmod} holds for  all $X, Y \subseteq V$ such that $X\cap Y\neq\emptyset$. A set function $b$ on $V$ is called \textit{submodular} if $-b$ is supermodular. 

A pair $X,Y$ of subsets of $V$ is called \textit{properly intersecting} if $X - Y, Y - X, X \cap Y\neq\emptyset.$
We say that a subset $X$ of $V$ \textit{crosses} a partition $\mathcal{P}$ if $X$ intersects at least 2 members of $\mathcal{P}$.

Let $\mathcal{P}_1$ and $\mathcal{P}_2$ be partitions of $V$. Let $\mathcal{P}$ be the family consisting of the sets in $\mathcal{P}_1$ and $\mathcal{P}_2$. While  there exist properly intersecting sets $X$ and $Y$ in $\mathcal{P}$, we replace them by $X \cap Y$ and $X \cup Y$. This preserves the property that each vertex belongs to exactly two sets of $\mathcal{P}$. Furthermore, the final family $\mathcal{P}$ does not contain properly intersecting sets, therefore it can be partitioned into two partitions, one consisting of the maximal elements (we denote it by {\boldmath$\mathcal{P}_1 \sqcup \mathcal{P}_2$}) of $\mathcal{P}$ and the other consisting of the minimal elements (we denote it by {\boldmath$\mathcal{P}_1 \sqcap \mathcal{P}_2$}) of $\mathcal{P}$. Let us give a list of some  properties of the partitions we have just defined. 
\begin{align}
&	\text{For all intersecting $U_1 \in \mathcal{P}_1$, $U_2 \in \mathcal{P}_2,$ there exists $Y\in \mathcal{P}_1 \sqcap \mathcal{P}_2$ such that $U_1 \cap U_2 \subseteq Y$,} \label{kvkjvk1}	\\
&	\text{For every $U \in \mathcal{P}_1 \cup \mathcal{P}_2$, there exists $Y \in \mathcal{P}_1 \sqcup \mathcal{P}_2$ such that $U \subseteq Y$,} \label{kvkjvk2}	\\
&	|\mathcal{P}_1 \sqcup \mathcal{P}_2| + |\mathcal{P}_1 \sqcap \mathcal{P}_2| = |\mathcal{P}_1| + |\mathcal{P}_2|,\label{kvkjvk3}\\
&	|\mathcal{P}_1 \sqcap \mathcal{P}_2| \ge \max\{|\mathcal{P}_1|, |\mathcal{P}_2|\} \ge \min\{|\mathcal{P}_1|, |\mathcal{P}_2|\} \ge |\mathcal{P}_1 \sqcup \mathcal{P}_2|.\label{usecl1}
\end{align}

\subsection{Digraphs and dypergraphs}

Let $D = (V, A)$ be a directed graph, shortly \textit{digraph}. We also denote its vertex set by $\bm{V(D)}$ and its arc set by $\bm{A(D)}$. For  $X \subseteq V$ and $F \subseteq A$, 
the \textit{in-degree} of $X$ in $F$, denoted by $\bm{d_F^-(X)}$, is the number of arcs in $F$ entering $X$. For $F \subseteq A$ and a subpartition $\mathcal{P}$ of $V$, we denote by $\bm{e_F(\mathcal{P})}$ the number of arcs in $F$ that enter some member of $\mathcal{P}$. We say that a directed graph $(U, F)$ is a \textit{branching} with root set $S$, shortly $S$-\textit{branching},  if $S \subseteq U$ and if there exists a unique path from $S$ to  every $u \in U$.
When $S = \{s\}$, we call it an \textit{arborescence} with root $s$. A subgraph $D'$ of a digraph $D$ is said to be \textit{spanning} if $V(D') = V(D)$.

Let $\mathcal{D}= (V, \mathcal{A})$ be a directed hypergraph, shortly \textit{dypergraph}. We also denote its vertex set by $\bm{V(\mathcal{D})}$ and  its hyperarc set by $\bm{\mathcal{A}(\mathcal{D})}$, where a hyperarc has exactly one head and at least one tail. For $X \subseteq V$ and $\mathcal{F} \subseteq \mathcal{A}$, the \textit{in-degree}  of $X$ in $\mathcal{F}$, denoted by $\bm{d_\mathcal{F}^-(X)}$, is the number of hyperarcs in $\mathcal{F}$ entering  $X$. 
We denote by {\boldmath$V_{\ge 1}(\mathcal{D})$} the set of vertices whose in-degree is at-least $1$ in $\mathcal{D}.$
For  a subpartition $\mathcal{P}$ of $V$, we denote by $\bm{e_\mathcal{F}(\mathcal{P})}$ the number of hyperarcs in $\mathcal{F}$ that enter some member of $\mathcal{P}$. 
By \textit{trimming} a hyperarc $X$ in $\mathcal{A}$ we mean the operation that  replaces $X$  by an arc $yx$, where $x$ is the head of $X$ and $y$ is one of the tails of $X$.
We say that a dypergraph $\mathcal{B}=(U,\mathcal{F})$ is a \textit{hyperbranching} with root set $S$, shortly $S$-\textit{hyperbranching}, if $S \subseteq U$ and if the hyperarcs in $\mathcal{F}$  can be trimmed to obtain an arc set $F$ such that the digraph $(V_{\ge 1}(\mathcal{B})\cup S,F)$ is an $S$-branching. The \textit{core} of an $S$-hyperbranching $\mathcal{B}$ is the vertex set $V_{\ge 1}(\mathcal{B})\cup S.$ 
When $S=\{s\}$, we call $\mathcal{B}$ a \textit{hyperarborescence} with root $s$. 
A subdypergraph $\mathcal{B}$ of $\mathcal{D}$ is a \textit{spanning} $S$-hyperbranching if it is an $S$-hyperbranching whose core  is $V.$ 

\subsection{Graphs and hypergraphs}

Let $G = (V, E)$ be an undirected graph, shortly \textit{graph}. We also denote its vertex set by $\bm{V(G)}$ and  its edge set by $\bm{E(G)}$. 
For a set $\mathcal{K}$ of edge-disjoint graphs, we denote the union of their edge sets by {\boldmath$E(\mathcal{K})$}.
For $X \subseteq V$ and $F \subseteq E$, we denote by $\bm{d_F(X)}$ the number of edges in $F$ entering $X$. We denote by $\bm{\mathcal{P}(G)}$ the partition of $V$ consisting of the connected components of $G$ and by $\bm{c(G)} =\bm{c(V,E)} = |\mathcal{P}(G)|$ the number of connected components of $G.$ For $F \subseteq E$, let $\bm{\mathcal{P}(F)}=\mathcal{P}((V,F))$. 
For $F \subseteq E$ and a partition $\mathcal{P}$ of $V$, we denote by $\bm{e_F(\mathcal{P})}$ the number of edges in $F$ that enter some member of $\mathcal{P}$. A subgraph $G'$ of a graph $G$  is said to be \textit{spanning} if $V(G') = V(G)$. 
A graph $G$ is called \textit{bipartite} if there exists a bipartition $\{A,B\}$ of its vertex set such that every edge of $G$ connects a vertex of $A$ to a vertex of $B,$ it is then denoted by {\boldmath$(A,B;E)$}. For bipartite graph $G=(A,B;E)$ and  $X\subseteq A,$ we denote by {\boldmath$\Gamma (X)$}  the set of vertices that are connected to at least one vertex in $X.$
The graph $G$ is called a \textit{forest} if it has no cycle. A connected forest  is called a \textit{tree}.
We say that a spanning forest $A$ \textit{crosses} another spanning forest $B$ if  there is a connected component of $A$ intersecting at least two connected components of $B$.

Let $\mathcal{G} = (V, \mathcal{E})$ be a hypergraph. We also denote its vertex set by $\bm{V(\mathcal{G})}$ and  its hyperedge set by $\bm{\mathcal{E}(\mathcal{G})}$, where a hyperedge consists of at least two vertices of $V.$ For $X \subseteq V$ and $\mathcal{F} \subseteq \mathcal{E}$, we denote by $\bm{d_\mathcal{F}(X)}$ the number of hyperedges in $\mathcal{F}$ entering $X$. 
Further, for  a subpartition $\mathcal{P}$ of $V$, we denote by $\bm{e_\mathcal{F}(\mathcal{P})}$ the number of hyperedges in $\mathcal{F}$ that enter some member of $\mathcal{P}$.
 By \textit{trimming} a hyperedge $X$ in $\mathcal{E}$ we mean the operation that  replaces $X$  by an edge between two different vertices in $X$. 
 By \textit{orienting} a hyperedge $X$ of $\mathcal{E}$ we mean the operation that chooses a vertex of $X$ to be its head and hence $X$ becomes a hyperarc.
A hypergraph $\mathcal{T}$ is called  a \textit{hyperforest} with root set $S$, shortly \textit{$S$-hyperforest}, if the hyperedges of $\mathcal{T}$ can be oriented to obtain an $S$-hyperbranching. 
When $S=\{s\}$, we call $\mathcal{T}$ a \textit{hypertree} with root $s$. 
An $S$-hyperforest $\mathcal{T}$ in $\mathcal{G}$ is said to be \textit{spanning} if it can be oriented to a spanning $S$-hyperbranching.

\subsection{Packing}

By a \textit{packing} of hyperbranchings (rooted hyperforests) in a dypergraph $\mathcal{D}$ (hypergraph $\mathcal{G}$), we mean a set of hyperarc-disjoint hyperbranchings (hyperedge-disjoint rooted hyperforests) in $\mathcal{D}$ (in $\mathcal{G}$, respectively). 
Let $h \in \mathbb{Z}_{> 0}$.
A packing of hyperbranchings is said to be \textit{$h$-regular}  if every vertex of $\mathcal{D}$ belongs to exactly $h$ cores of the hyperbranchings. 
In the case of digraphs this is equivalent to saying that  each vertex belongs to $h$ of the branchings.
A packing of rooted hyperforests is said to be \textit{$h$-regular} if it can be oriented to obtain an $h$-regular packing of hyperbranchings.
In the case of undirected graphs this is equivalent to saying that  each vertex belongs to $h$ of the forests.
A packing of $k$ $S_i$-hyperbranchings or $k$  rooted $S_i$-hyperforests is said to be \textit{$(\ell, \ell')$-bordered} for $\ell, \ell' : \mathbb{Z}_k \rightarrow \mathbb{Z}_{>0}$ if and only if \eqref{borderedcond} holds and is said to be \textit{$(\alpha, \beta)$-limited} for $\alpha, \beta \in \mathbb{Z}_{>0}$ if and only if  \eqref{limitedcond} holds
\begin{alignat}{2}
	\ell'(i) &\ge &&|S_i| \ge \ell(i) \qquad \text{for every } 1 \le i \le k, \label{borderedcond}\\
	\beta &\ge \sum_{i = 1}^k &&|S_i| \ge \alpha. \label{limitedcond}
\end{alignat}

\subsection{Matroid theory}\label{matroidsection}

We use the usual notions from matroid theory. Let $S$ be a finite ground set. A function $r:2^S\rightarrow \mathbb Z_{\ge 0}$ is called the \textit{rank function} of the \textit{matroid} $\bm{{\sf M}} = (S, r)$ if and only if $r(X) \le |X|$ for every $X \subseteq S$, $X \subseteq Y \subseteq S$ implies $r(X) \le r(Y)$, and $r$ is submodular.
A subset $X$ of $S$ is called an \textit{independent set} of ${\sf M}$  if $r_{\sf M}(X)=|X|$.

\begin{itemize}	
	\item[(1)] For an undirected graph $G = (V, E)$, the \textit{graphic matroid} $\bm{{\sf M}_G} = (E, r_G)$ is defined such that $r_G(F) = |V| - c(V, F)$ for any subset $F$ of $E$.
	It is well-known that the set of edge sets of spanning forests in $G$ is exactly the set of independent sets of ${\sf M}_G.$
	\item[(2)] For   a matroid ${\sf M} = (S, r)$ and $\ell \in \mathbb{Z}_{\ge 0}$, the \textit{truncated matroid} of ${\sf M}$ at $\ell$ is ${\sf M}' = (S, r')$, where $r'(X) = \min\{r(X), \ell\}$ for every $X\subseteq S.$
	\item[(3)] For $k$ matroids ${\sf M}_1 = (S, r_1), \ldots, {\sf M}_k = (S, r_k)$  on the same ground set $S$, we define \textit{the sum} of ${\sf M}_1, \ldots, {\sf M}_k$ to be the matroid ${\sf M}$ such that $X$ is independent in  {\sf M} if and only if $X$ can be partitioned into $X_1, \ldots, X_k$ such that $X_i$ is independent in ${{\sf M}_i}$ for every $1 \le i \le k$.
\end{itemize}

\begin{thm}[Edmonds, Fulkerson \cite{EF}]\label{partitionthm}
    The rank function $r$ of the sum matroid ${\sf M}$ of $k$ matroids ${\sf M}_i = (S, r_i)$ is given by the following formula
    \begin{equation*}
        r(Z) = \min_{X \subseteq Z}\left\{|Z - X| + \sum_{i = 1}^{k} r_i(X)\right\} \hskip .5truecm \text{ for every  } Z\subseteq S.
    \end{equation*}
\end{thm}

\section{Preliminary results}\label{preresults}

In this section we provide preliminary results to be applied later.
\medskip

We start with a simple technical claim.

\begin{cl}\label{submodpart}
Let $a_1,a_2,a_1',a_2',\ell\in\mathbb{Z}_{>0}$ such that $a_1+a_2=a_1'+a_2'$ and $\min\{a_1,a_2\}\ge a'_2.$ Then 
\begin{eqnarray*} \label{bdkjdkjvjkek}
\min\{\ell,a_1\}+\min\{\ell,a_2\}\ge\min\{\ell,a'_1\}+\min\{\ell,a_2'\}.
\end{eqnarray*}
\end{cl}

\begin{proof}
We may assume without loss of generality that $a_1'\ge a_1\ge a_2\ge a_2'.$ 

If $\ell\ge a_1$, then $\min\{\ell,a_1\}+\min\{\ell,a_2\}=a_1+a_2=a_1'+a_2'\ge \min\{\ell,a'_1\}+\min\{\ell,a_2'\}.$  

If $a_1\ge\ell$, then $\min\{\ell,a_1\}+\min\{\ell,a_2\}\ge\ell+\min\{\ell,a'_2\}=\min\{\ell,a'_1\}+\min\{\ell,a_2'\}.$
\end{proof}

We need the submodularity of $e_{\mathcal{E}}$ on partitions.

\begin{Lemma}\label{submodeep}
	Let $\mathcal{G} = (V, \mathcal{E})$  be a hypergraph and $\mathcal{P}_1$ and $\mathcal{P}_2$  partitions of $V$. The following holds
\begin{equation}\label{submodeepeq}
	e_{\mathcal{E}}(\mathcal{P}_1) + e_{\mathcal{E}}(\mathcal{P}_2) \ge e_{\mathcal{E}}(\mathcal{P}_1 \sqcap \mathcal{P}_2) + e_{\mathcal{E}}(\mathcal{P}_1 \sqcup \mathcal{P}_2).
\end{equation}
\end{Lemma}

\begin{proof}
The result will easily follow  from the following claim.	

\begin{cl}\label{lkbzlkbl1} 
The following hold for all $X \in \mathcal{E}$.

(a) If $X$ crosses $\mathcal{P}_1 \sqcup \mathcal{P}_2$ then it crosses both $\mathcal{P}_1$ and $\mathcal{P}_2$.

(b) If $X$ crosses $\mathcal{P}_1 \sqcap \mathcal{P}_2$ then it crosses either $\mathcal{P}_1$ or $\mathcal{P}_2$. 
\end{cl}

\begin{proof}
(a) Suppose that $X$ leaves $Y \in \mathcal{P}_1 \sqcup \mathcal{P}_2$, that is, $\emptyset \neq X \cap Y \neq X$. Let $U \in \mathcal{P}_1$ be such that $X \cap Y \cap U \neq \emptyset$ ($U$ exists because $\mathcal{P}_1$ is a partition of $V$). By \eqref{kvkjvk2} and $Y \cap U \neq \emptyset$, we have $U \subseteq Y$, thus $X$ also leaves $U$, that is, $X$ crosses $\mathcal{P}_1$. The same holds for $\mathcal{P}_2$.

(b) Suppose that $X$ leaves $Y \in \mathcal{P}_1 \sqcap \mathcal{P}_2$. Since $\mathcal{P}_1$ and $\mathcal{P}_2$ are partitions of $V$, there exist $U_1 \in \mathcal{P}_1$ and $U_2 \in \mathcal{P}_2$ such that $U_1 \cap U_2 \cap X \cap Y \neq \emptyset$.
By \eqref{kvkjvk1}, we have $U_1 \cap U_2 \subseteq Y$ which implies that $X$ leaves $U_1 \cap U_2$ and thus leaves either $U_1$ or $U_2$, that is, $X$ crosses either $\mathcal{P}_1$ or $\mathcal{P}_2$.
\end{proof}

For any $X \in \mathcal{E}$, if $e_{\{X\}}(\mathcal{P}_1 \sqcup \mathcal{P}_2) = 1$ then, by Claim \ref{lkbzlkbl1}(a), we have $e_{\{X\}}(\mathcal{P}_1) + e_{\{X\}}(\mathcal{P}_2) = 2 \ge e_{\{X\}}(\mathcal{P}_1 \sqcap \mathcal{P}_2) + e_{\{X\}}(\mathcal{P}_1 \sqcup \mathcal{P}_2)$. If $e_{\{X\}}(\mathcal{P}_1 \sqcup \mathcal{P}_2) = 0$ then, by Claim \ref{lkbzlkbl1}(b), we have $e_{\{X\}}(\mathcal{P}_1) + e_{\{X\}}(\mathcal{P}_2) \ge e_{\{X\}}(\mathcal{P}_1 \sqcap \mathcal{P}_2) = e_{\{X\}}(\mathcal{P}_1 \sqcap \mathcal{P}_2) + e_{\{X\}}(\mathcal{P}_1 \sqcup \mathcal{P}_2)$.
Hence \eqref{submodeepeq} follows.
\end{proof}

We need the following negative result on packing trees.

\begin{thm}[Kirkpatrick, Hell \cite{kirkpat}]\label{jbejvejhfvefvoidouevou}
Let $G = (V, E)$ be a graph  and $h, k, \ell \in \mathbb Z_{>0}$ with $\ell\ge 2$. It is NP-complete to decide whether there  exists an $h$-regular packing of $k$ trees in $G$ each containing $\ell$ edges, even for $h=1$ and $\ell = 2$.
\end{thm}

From this we derive the same results for forests and branchings. They justify why we do not consider in our packing problems the number of edges  and arcs.

\begin{thm}\label{hkbrell}
Let $D$ be a digraph, $G$ a graph, $h, k\in\mathbb Z_{>0}$, and $\ell: \mathbb{Z}_k \rightarrow  \mathbb{Z}_{>0}$.
\begin{enumerate}[label=(\alph*)]
	\item It is NP-complete to decide whether there exists an $h$-regular packing of $k$ forests in $G$  containing $\ell(1),\dots,\ell(k)$ edges, even for $h = 1$ and $\ell(i) = 2$ for every $1\le i\le k.$\label{npcthmb}
	\item It is NP-complete to decide whether there exists an $h$-regular packing of $k$ branchings in $D$  containing $\ell(1),\dots,\ell(k)$ arcs, even for $h = 1$ and $\ell(i) = 2$ for every $1\le i\le k.$\label{npcthma}
\end{enumerate}
\end{thm}

\begin{proof}
\ref{npcthmb} follows from Theorem~\ref{jbejvejhfvefvoidouevou}. Indeed, for a graph $G$, there exists a $1$-regular packing of $\frac{|V|}{3}$ forests in $G$ each containing $2$ edges if and only there exists a $1$-regular packing of $\frac{|V|}{3}$ trees in $G$ each containing $2$ edges.

\ref{npcthma} follows from \ref{npcthmb}.
First, the problem of finding an $h$-regular packing of $k$ branchings each containing $2$ arcs in a digraph is clearly in NP.
Now we show that it is NP-hard, by performing a reduction from the problem of \ref{npcthmb}.
Indeed, from a graph $G=(V,E)$, we can construct a digraph $D=(V,A)$  by adding arcs $uv$ and $vu$ to $A$ for every $uv \in E$. Clearly, the size of $D$ is polynomial in the size of $G$. Finally, there exists a $1$-regular packing of $\frac{|V|}{3}$ branchings each containing $2$ arcs in $D$ if and only there exists a $1$-regular packing of $\frac{|V|}{3}$ forests each containing $2$ edges in $G$.
\end{proof}

\medskip

A possible extension of an $h$-regular packing, would be to take $h$ to be a function $h : V \rightarrow \mathbb{Z}_{\ge 0}$ and ask for each $v \in V$ to be in exactly $h(v)$ subgraphs of the packing (when $h$ is constant, this reduces to the definition we gave earlier). The following negative results show that, with this extension, even the simplest problems become NP-complete.

\begin{thm}\label{funcregnpcomplete}
	Let $D=(V,A)$ be a digraph, $G=(V,E)$ a graph, $k \in \mathbb Z_{>0}$ and $h : V \rightarrow \mathbb{Z}_{\ge 0}$.
	\begin{enumerate}[label=(\alph*)]
		\item It is NP-complete to decide whether there exists a packing of $k$ arborescences in $D$ such that every $v \in V$ belongs to exactly $h(v)$ arborescences of the packing, even for $h : V \rightarrow \{1, 2\}$.\label{funcregd}
		\item It is NP-complete to decide whether there exists a packing of $k$ trees in $G$ such that every $v \in V$ belongs to exactly $h(v)$ trees of the packing, even for $h : V \rightarrow \{1, 2\}$.\label{funcregund}
	\end{enumerate}
\end{thm}

\begin{proof} \ref{funcregd} We call {\sc Frpa} the problem of deciding wether an instance $(D, h, k)$ admits a packing of $k$ arborescences such that every $v \in V$ belongs to exactly $h(v)$ arborescences of the packing. 
This decision problem is clearly in NP. We give a reduction from Monotone Not-all-equal-3SAT ({\sc Mnae3sat}). Given a set $X$ of boolean variables and a formula consisting of a set $\mathcal{C}$ of clauses each containing 3 distinct variables, none of which are negated, it is NP-complete to decide whether  there exists a truth assignment to the variables of $X$ such that every clause in $\mathcal{C}$ contains at least one true and at least one false literal, Schaefer \cite{schaefer}.

Let $(\bm{X}, \bm{\mathcal{C}})$ be an instance of {\sc Mnae3sat}. We define an instance of {\sc Frpa} as follows. Let $\bm{D} =(V, A)$ be the digraph where $\bm{V} = V_X \cup V_{\mathcal{C}} \cup \{s, v_T, v_F\}$, $\bm{V_X} = \{v_{x}: x\in X\}, \bm{V_{\mathcal{C}}} = \{v_{C}: C\in {\mathcal{C}}\}, \bm{A} = A_1 \cup A_2 \cup A_3, \bm{A_1} = \{v_T v_x, v_F v_x: x \in X\}, \bm{A_2} = \{v_x v_C,v_y v_C,v_z v_C: x \vee y \vee z = C \in \mathcal{C}\}$ and $\bm{A_3} = \{sv_T, sv_F\}$. Let $\bm{h}: V \rightarrow \mathbb Z_{\ge 0}$ be the following function: $h(v_x) = 1$ for every $x \in X$, $h(v_C) = 2$ for every $C \in \mathcal{C}$, $h(v_T) = h(v_F) = 1$, $h(s) = 2$ and $k = 2$. 
Note that the size of $(D,h,k)$ is clearly polynomial in the size of $(X,\mathcal{C}).$
We now show that $(X,\mathcal{C})$ is a positive instance of {\sc Mnae3sat} if and only if $(D, h, k)$ is a positive instance of {\sc Frpa}.
		
Let us take a truth assignment $\varphi$ to the variables of $X$ such that every clause in $\mathcal{C}$ contains at least one true and at least one false literal. To each $C\in\mathcal{C}$ we can hence choose a variable {\boldmath$x_C^T$} in $C$ such that $\varphi(x_C^T) = \text{true}$ and a variable {\boldmath$x_C^F$} in $C$ such that $\varphi(x_C^F) = \text{false}$. We construct the two required arborescences in $D$ as follows: 
$$\bm{B_T} = (V_{\mathcal{C}}\cup \{v_x: \varphi(x) = \text{true}\} \cup v_T \cup s, sv_T \cup \{v_T v_x: \varphi(x) = \text{true}\} \cup \{v_{x_C^T} v_C: C \in \mathcal{C}\}),$$  
$$\bm{B_F} = (V_{\mathcal{C}}\cup \{v_x: \varphi(x) = \text{false}\} \cup v_F\cup s, sv_F\cup\{v_Fv_x:\varphi(x) = \text{false}\}\cup\{v_{x_C^F}v_C:C\in\mathcal{C}\}).$$ 
It is clear that $B_T$ and $B_F$ are arc-disjoint $s$-arborescences. Moreover, by construction, each vertex $v$ of $D$ belongs to $h(v)$ of them. Therefore $(D, h, k)$ is a positive instance of {\sc Frpa}.
		
Let us now take a packing of two arborescences in $D$ such that every vertex $v\in V$ belongs to exactly $h(v)$ of them. Since $d_{A}^+(s) = 2 = h(v_C)$ for any $C \in \mathcal{C}$ and $d_A^-(s) = 0$, the packing consists of two $s$-arborescences and, $sv_T$ and $sv_F$ do not belong to the same arborescence. We denote these arborescences by $B_T$ and $B_F$, respectively. Since $h(v_x) = 1$ for every $x \in X$, $v_x$ belongs to $B_T$ or $B_F$. We hence define a truth assignment $\varphi$ as follows: let $\varphi(x)$ be true if and only if $v_x \in V(B_T)$. Since $h(v_C) = 2$ for every $C \in \mathcal{C}$, $v_C$ belongs to $B_T$ and $B_F$, so $C$ contains at least one true and at least one false literal. Then $(X, \mathcal{C})$ is a positive instance of {\sc Mnae3sat}.
This concludes the proof of \ref{funcregd}. 
\medskip

One can apply a  similar reduction to show \ref{funcregund}.
\end{proof}

Finally, we also need the following two results. The first one can be obtained by combining Theorems 16 and 17 in \cite{BF3}, while the second one  is a special case of Corollary 1 of \cite{FKLSzT}.

\begin{thm}[B\'erczi, Frank \cite{BF3}]\label{BFBG}
Let $S$ and $T$ be disjoint sets, $p$  a positively intersecting supermodular set function on $T$,
  $f,g: T\cup S\rightarrow \mathbb{Z}_{\ge 0}$ and $\alpha, \beta\in\mathbb{Z}_{\ge 0}$ such that $f\le g$ and $\alpha\le\beta.$
There exists a simple bipartite graph $G=(S,T,E)$ such that 
\begin{eqnarray*}
|\Gamma_E(Y)|		&	\ge 	&	p(Y) \hskip .2truecm \text{ for every } Y\subseteq T,\\
f(v)\hskip .2truecm \le \hskip .45truecm d_E(v)\hskip .15truecm 		&	\le 	&	g(v) \hskip .3truecm \text{ for every } v\in T\cup S,\\
\alpha\hskip .2truecm \le\hskip .6truecm  |E|\hskip .4truecm 		&	\le	&	\beta
\end{eqnarray*}
if and only if  for all $X\subseteq S, Y\subseteq  T$ and subpartition $\mathcal{P}$ of $T-Y,$ we have
\begin{eqnarray}
f(Y)-|X||Y|+\sum_{P\in\mathcal{P}}p(P)-|X||\mathcal{P}|			&	\le	&	g(S-X),\label{41}\\
f(X)-|X||Y|+\sum_{P\in\mathcal{P}}p(P)-|X||\mathcal{P}|			&	\le	&	g(T-Y),\label{42}\\
\alpha-|X||Y|+\sum_{P\in\mathcal{P}}p(P)-|X||\mathcal{P}|			&	\le	&	g((S-X)\cup(T-Y)),\label{50}\\
f(X\cup Y)-|X||Y|+\sum_{P\in\mathcal{P}}p(P)-|X||\mathcal{P}|		&	\le	&	\beta.\label{51}
\end{eqnarray}

\end{thm}

\begin{thm}[Fortier et al. \cite{FKLSzT}]\label{blieioeibiuzg}
Let $\mathcal{D}=(V,\mathcal{A})$ be a dypergraph,  ${\cal S}$  a family of subsets of $V$ and $h\in\mathbb Z_{>0}$. 
There exists an $h$-regular packing of $S$-hyperbranching $(S\in {\cal S})$ in $\mathcal{D}$  if and only if 
	\begin{eqnarray*}  
		|{\cal S}_X|+d^-_\mathcal{A}(X)  	&	\geq 	&	h \hskip .5truecm \text{ for every  $\emptyset\neq X\subseteq V,$}\\
		|{\cal S}_v|			&	\le	&	h \hskip .5truecm \text{ for every } v\in V.
	\end{eqnarray*}
\end{thm}

\section{Prior results}\label{knownresults}

The problems of packing trees and packing arborescences have been studied for a long time in different scopes. The first results motivating this paper are due to Nash-Williams~\cite{NW}, Tutte~\cite{Tu} and Frank~\cite{FA78}. In their papers of 1961, Nash-Williams~\cite{NW} and Tutte~\cite{Tu} simultaneously showed the following result about packing spanning trees in undirected graphs.

\begin{thm}[Nash-Williams~\cite{NW}, Tutte~\cite{Tu}]\label{tuttethmrem} 
Let $G=(V,E)$ be a graph and $k\in \mathbb Z_{>0}.$
There exists a packing of $k$ spanning trees in $G$  if and only if  
	\begin{alignat*}{2}
		e_E(\mathcal{P}) \ge k(|\mathcal{P}| - 1) \qquad \text{for every partition } \mathcal{P} \text{ of } V.
	\end{alignat*}  
\end{thm}

Almost two decades later, Frank~\cite{FA78} showed the directed counterpart of Theorem~\ref{tuttethmrem}.

\begin{thm}[Frank~\cite{FA78}]\label{frankarborescences} 
Let $D = (V, A)$ be a digraph and $k \in \mathbb{Z}_{>0}$.
There exists a packing of $k$  spanning arborescences in $D$ if and only if
\begin{alignat*}{2}
e_A(\mathcal{P}) \ge k(|\mathcal{P}| - 1) \qquad \text{for every subpartition } \mathcal{P} \text{ of } V.
\end{alignat*}
\end{thm}
One can see that the conditions of both theorems are very similar. It is nonetheless of importance to note that the condition for undirected graphs is about partitions while the condition for digraphs is about subpartitions.
This small difference is actually present between all the results of this paper on undirected graphs and their counterparts for digraphs, the conditions are the same, but the undirected case is with partitions and the directed with subpartitions. 
\medskip

More recently, these seminal results were generalized to hypergraphs and dypergraphs.

Theorem \ref{tuttethmrem} was generalized to hypergraphs in~\cite{fkk}.

\begin{thm} [Frank, Kir\'aly, Kriesell \cite{fkk}]\label{tuttethmremhyp} 
Let $\mathcal{G}=(V,\mathcal{E})$ be a hypergraph and $k\in \mathbb{Z}_{>0}.$
There exists a packing of $k$ spanning hypertrees in $\mathcal{G}$  if and only if  
	\begin{alignat*}{2}
		e_{\mathcal{E}}(\mathcal{P}) &\ge k(|\mathcal{P}| - 1) \qquad \text{for every partition } \mathcal{P} \text{ of } V.
	\end{alignat*}  
\end{thm}

An extension of Theorem \ref{frankarborescences} to dypergraphs can be obtained from either Fortier et al.  \cite[Corollary 1]{FKLSzT} or H\"orsch and Szigeti \cite[Theorem 8]{HSz5}.

\begin{thm}[\cite{FKLSzT,HSz5}]\label{freehyperarborescences} 
Let $\mathcal{D}=(V,\mathcal{A})$ be a dypergraph and $k\in \mathbb{Z}_{>0}.$
 There exists a packing of   $k$  spanning hyperarborescences in $\mathcal{D}$ if and only if 
 	\begin{alignat*}{2}
		 e_\mathcal{A}(\mathcal{P}) \ge k(|\mathcal{P}| - 1) \qquad \text{for every subpartition } \mathcal{P} \text{ of } V.
	\end{alignat*}
\end{thm}
\medskip

B\'erczi and Frank~\cite{BF3} gave a major generalization of Theorem~\ref{frankarborescences}.

\begin{thm}[B\'erczi, Frank \cite{BF3}]\label{BFmain}
	Let $D = (V, A)$ be a digraph, $k, \alpha, \beta \in \mathbb{Z}_{>0}$, and $\ell, \ell': \mathbb{Z}_k \rightarrow \mathbb{Z}_{>0}$ such that
	\begin{alignat}{3}
		\ell'(\mathbb{Z}_k) &\ge \beta \ge \alpha \ge \ell(\mathbb{Z}_k), \label{totnecessary}\\
		|V| &\ge \ell'(i) \ge \ell(i) \qquad\qquad \text{for every } 1 \le i \le k. \label{indivnecessary}
	\end{alignat}
	There exists an $(\ell, \ell')$-bordered $(\alpha, \beta)$-limited packing of $k$ spanning  branchings in $D$ if and only if
	\begin{alignat}{3}
	\beta - \ell(\mathbb{Z}_k) + \ell_{|\mathcal{P}|}(\mathbb{Z}_k) + e_A(\mathcal{P}) &\ge k|\mathcal{P}| \qquad &&\text{for every subpartition } \mathcal{P} \text{ of } V, \\
	\ell'_{|\mathcal{P}|}(\mathbb{Z}_k) + e_A({\cal P}) &\ge k|\mathcal{P}| \qquad &&\text{for every subpartition } \mathcal{P} \text{ of } V.
	\end{alignat}
\end{thm}

When $\ell(i) = \ell'(i) = 1$ for $i = 1, \dots, k$ and $\alpha = \beta = k$,  Theorem \ref{BFmain} reduces to Theorem \ref{frankarborescences}.

\medskip

In Theorem~\ref{BFmain}, we assume that~\eqref{totnecessary} holds. This assumption is actually not restrictive at all. In fact, as $\ell'(\mathbb{Z}_k)$ is an upper bound on the total number of roots in the packing, adding $\beta > \ell'(\mathbb{Z}_k)$ as a new upper bound would be meaningless. The same reasoning can be applied to $\ell(\mathbb{Z}_k)$ and $\alpha$ which are both lower bounds on the total number of roots in the packing. Finally, it is necessary to have $\beta \ge \alpha$, thus we assume it holds  not to be burdened by the trivial case of $\beta < \alpha$ later.

We also assume that~\eqref{indivnecessary} holds. It is easy to see $|V|$ is a natural upper bound on the number of roots in any branching. This is why assuming that $|V| \ge \ell'(i)$ for every $1 \le i \le k$ is not restrictive at all. Finally, as $\ell'(i) \ge \ell(i)$ is necessary, we may assume it holds without loss of generality.

\medskip

In this paper we generalize Theorem \ref{BFmain} to regular packings in dypergraphs. The main goal of this paper is to provide its undirected counterpart. 

\section{New results}\label{newresults}

The first  contribution of this paper is the following generalization of Theorem~\ref{BFmain}.

\begin{thm}\label{jjjnbbbnajkhypregdir}
Let $\mathcal{D} = (V, \mathcal{A})$ be a dypergraph, $h, k, \alpha, \beta \in \mathbb Z_{>0}$ and $\ell, \ell': \mathbb{Z}_k \rightarrow  \mathbb{Z}_{>0}$ such that~\eqref{totnecessary} and~\eqref{indivnecessary} hold. There exists an $h$-regular $(\ell, \ell')$-bordered $(\alpha, \beta)$-limited packing of $k$  hyperbranchings in $\mathcal{D}$ if and only if
\begin{alignat}{3}
	h|V| &\ge \alpha, \label{hValpha}\\
	\beta - \ell(\mathbb{Z}_k) + \ell_{|\mathcal{P}|}(\mathbb{Z}_k) + e_\mathcal{A}(\mathcal{P}) &\ge h|\mathcal{P}| \qquad &&\text{for every subpartition } \mathcal{P} \text{ of } V, \label{konfeiobfiueguye1mregdir}\\
	\ell'_{|\mathcal{P}|}(\mathbb{Z}_k) + e_\mathcal{A}(\mathcal{P}) &\ge h|\mathcal{P}| \qquad &&\text{for every subpartition } \mathcal{P} \text{ of } V. \label{konfeiobfiueguye2mregdir}
	\end{alignat}
\end{thm}

When $\mathcal{D}$ is a digraph and $h = k$, Theorem \ref{jjjnbbbnajkhypregdir} reduces to Theorem \ref{BFmain}.

Theorem~\ref{jjjnbbbnajkhypregdir} is proved in Subsection \ref{lkdlkcbdlkcdckl2hyper}. This proof uses almost exclusively the tools and technics employed by B\'erczi and Frank in their article~\cite{BF3}.
\medskip

The  main contribution of this article is the undirected counterpart of Theorem~\ref{jjjnbbbnajkhypregdir}. 
In order to prove it, we need to introduce some intermediate theorems.

\begin{thm}\label{kjjhvhgu}
Let $G = (V, E)$ be a graph, $k \in \mathbb{Z}_{>0}$ and $\ell : \mathbb{Z}_k \rightarrow \mathbb{Z}_{>0}$.
 There exists a packing of $k$ spanning forests in $G$ with $\ell(1),\dots,\ell(k)$ connected components if and only if
	\begin{alignat}{3}
		|V| &\ge \ell(i) \qquad &&\text{for every } 1 \le i \le k, \label{bjehvde}\\
		\ell_{|\mathcal{P}|}(\mathbb{Z}_k) + e_E(\mathcal{P}) &\ge k|\mathcal{P}| \qquad &&\text{for every partition } \mathcal{P} \text{ of } V.
	\end{alignat}
\end{thm}

When $\ell(i) = 1$ for  $i = 1,\dots, k$, Theorem \ref{kjjhvhgu} reduces to Theorem \ref{tuttethmrem}.

Actually, Peng, Chen and Koh \cite{PCK} gave a proof of Theorem \ref{kjjhvhgu} when all $\ell(i)$'s are equal.

Theorem~\ref{kjjhvhgu} is a straightforward application of matroid theory. For the sake of completeness, we provide its proof in Subsection~\ref{pojfpohjfop}.

We mention that the directed counterpart of Theorem \ref{kjjhvhgu} is a special case of Theorem \ref{BFmain}, however their natural extension  to mixed graphs does not hold.
\medskip

The following result and its proof come from \cite{hopp}.

\begin{thm}\label{kbjelciueundi}
Let $G = (V, E)$ be a graph, $h, k \in \mathbb{Z}_{>0}$ and  $\ell: \mathbb{Z}_k \rightarrow  \mathbb{Z}_{>0}$.
There exists an $h$-regular packing of $k$ forests in $G$ with $\ell(1), \dots, \ell(k)$ connected components if and only if~\eqref{bjehvde} holds and 
\begin{alignat}{3}
	h|V| &\ge \ell(\mathbb{Z}_k), \label{kjdjhzvdh}\\
	\ell_{|\mathcal{P}|}(\mathbb{Z}_k) + e_{E}(\mathcal{P}) &\ge h|\mathcal{P}| \qquad &&\text{for every partition } \mathcal{P} \text{ of } V. \label{jvjvkjbkhlhio} 
	\end{alignat}
\end{thm}

For $h=k$, Theorem \ref{kbjelciueundi} reduces to Theorem \ref{kjjhvhgu}.
The necessity  of Theorem \ref{kbjelciueundi} follows from the necessity  of its  extension, Theorem \ref{livzdljzvdmz}.
The sufficiency of Theorem \ref{kbjelciueundi} is proved in Subsection \ref{subsecproofhopp}. This is the main proof of the present paper, it allows us to obtain Theorem \ref{livzdljzvdmz}.
\medskip

\begin{thm}\label{livzdljzvdmz}
Let $G = (V, E)$ be a graph, $h, k, \alpha, \beta \in \mathbb{Z}_{>0}$, and $\ell, \ell': \mathbb{Z}_k \rightarrow \mathbb{Z}_{>0}$ such that~\eqref{totnecessary} and~\eqref{indivnecessary} hold.
There exists an $h$-regular $(\ell, \ell')$-bordered $(\alpha, \beta)$-limited packing of $k$ forests in $G$ if and only if~\eqref{hValpha} holds and
\begin{alignat}{3}
	\beta - \ell(\mathbb{Z}_k) + \ell_{|\mathcal{P}|}(\mathbb{Z}_k) + e_E(\mathcal{P}) &\ge h|\mathcal{P}|\qquad &&\text{for every partition } \mathcal{P} \text{ of } V, \label{konfeiobfiueguye1m}\\
	\ell'_{|\mathcal{P}|}(\mathbb{Z}_k) + e_E(\mathcal{P}) &\ge h|\mathcal{P}| \qquad &&\text{for every partition } \mathcal{P} \text{ of } V. \label{konfeiobfiueguye2m}
\end{alignat}
\end{thm}

For $\ell(i)=\ell'(i)$ for  $i=1,\dots, k$, $\alpha = \beta = \ell(\mathbb{Z}_k)$, Theorem \ref{livzdljzvdmz} reduces to Theorem \ref{kbjelciueundi}. 
The necessity  of Theorem \ref{livzdljzvdmz} follows from the necessity  of its hypergraphic version, Theorem \ref{jjjnbbbnajkhypreg}.
The sufficiency of Theorem \ref{livzdljzvdmz} is proved in Subsection \ref{subsecproofregforwithbounds}.
It will follow from  Theorem \ref{kbjelciueundi}. 
\medskip

The main result of this paper is the following undirected counterpart of Theorem \ref{jjjnbbbnajkhypregdir}, and generalizes all the results about undirected graphs and hypergraphs of this paper.

\begin{thm}\label{jjjnbbbnajkhypreg}
Let $\mathcal{G}=(V,\mathcal{E})$ be a hypergraph, $h, k\in \mathbb Z_{>0}$, $\ell,\ell': \mathbb{Z}_k \rightarrow \mathbb{Z}_{>0}$ such that~\eqref{totnecessary} and~\eqref{indivnecessary} hold. There exists an $h$-regular $(\ell, \ell')$-bordered $(\alpha, \beta)$-limited packing of $k$ rooted hyperforests in $\mathcal{G}$ if and only if~\eqref{hValpha} holds and
\begin{alignat}{3}
	\beta - \ell(\mathbb{Z}_k) + \ell_{|\mathcal{P}|}(\mathbb{Z}_k) + e_\mathcal{E}(\mathcal{P}) &\ge h|\mathcal{P}| \qquad &&\text{for every partition } \mathcal{P} \text{ of } V, \label{konfeiobfiueguye1mreg}\\
	\ell'_{|\mathcal{P}|}(\mathbb{Z}_k) + e_\mathcal{E}(\mathcal{P}) &\ge h|\mathcal{P}| \qquad &&\text{for every partition } \mathcal{P} \text{ of } V. \label{konfeiobfiueguye2mreg}
\end{alignat}
\end{thm}

If $\mathcal{G}$ is a graph then Theorem \ref{jjjnbbbnajkhypreg} reduces to Theorem \ref{livzdljzvdmz}.
Theorem \ref{jjjnbbbnajkhypreg} is proved in Subsection \ref{hdfoibdoibio}. It will follow from its graphic version, Theorem \ref{livzdljzvdmz}, by applying the operation trimming. 
\medskip

We mention that for $h = k$, Theorem~\ref{jjjnbbbnajkhypreg} reduces to a theorem that was proved in~\cite{MMSz} by applying the theory of generalized polymatroids.

\section{Proofs}\label{proofs}

In this section we prove our new results, Theorems  \ref{jjjnbbbnajkhypregdir} -- \ref{jjjnbbbnajkhypreg}.

\subsection{Proof of Theorem \ref{jjjnbbbnajkhypregdir}}\label{lkdlkcbdlkcdckl2hyper}

\begin{proof}
To prove the \textbf{necessity}, let $\{B_1, \ldots, B_k\}$ be an $h$-regular $(\ell, \ell')$-bordered $(\alpha, \beta)$-limited packing of $k$ hyperbranchings with respective root sets $\{S_1, \ldots, S_k\}$.  
By definition of $h$-regular, there can be at most $h|V|$ roots in total in the packing, therefore, by~\eqref{limitedcond} we have
\begin{equation*}
	h|V| \ge \sum_{i = 1}^k |S_i| \ge \alpha,
\end{equation*}
thus~\eqref{hValpha} holds.
To prove~\eqref{konfeiobfiueguye1mregdir} and~\eqref{konfeiobfiueguye2mregdir}, let $\mathcal{P}$ be a subpartition of $V$. 

 For $i\in \mathbb{Z}_k,$ we denote by $\mathcal{P}_{i} = \{X \in \mathcal{P}: X$ intersects the core of  $B_i\}$.
 Since for every $X\in\mathcal{P}_{i}$ that contains no root of $B_i$, there exists a hyperarc of $B_i$ that enters $X$, we get that 
\begin{equation}
	e_{\mathcal{A}(B_i)}(\mathcal{P}) \ge |\mathcal{P}_{i}| - |S_i|.\label{jkvsj3}
\end{equation}

Then, by \eqref{jkvsj3}, $\max\{0, a - b\}=a-\min\{a,b\}$ and \eqref{borderedcond}, we have
\begin{eqnarray}
	|S_i|-\ell(i)+e_{\mathcal{A}(B_i)}(\mathcal{P}) &\ge&  \max\{0, |\mathcal{P}_{i}| - \phantom{'}\ell(i)\} 
	\ge |\mathcal{P}_{i}| - \ell_{|\mathcal{P}|}(i),\label{jkvsj2}\\
	e_{\mathcal{A}(B_i)}(\mathcal{P}) &\ge& \max\{0, |\mathcal{P}_{i}| - \ell'(i)\}  \ge |\mathcal{P}_{i}| - \ell'_{|\mathcal{P}|}(i).\label{jkvsj1}
\end{eqnarray}

It follows, by \eqref{jkvsj2}, \eqref{limitedcond}, and \eqref{jkvsj1}, that we have 
\begin{eqnarray}
	e_{\mathcal{A}}(\mathcal{P}) & \ge &\sum_{i = 1}^k e_{\mathcal{A}(B_i)}(\mathcal{P}) \ge \sum_{i = 1}^{k} |\mathcal{P}_{i}| - \ell_{|\mathcal{P}|}(\mathbb{Z}_k)+\ell(\mathbb{Z}_k)-\beta,\label{prfun2'}\\
	e_{\mathcal{A}}(\mathcal{P}) &\ge& \sum_{i = 1}^k e_{\mathcal{A}(B_i)}(\mathcal{P}) \ge \sum_{i = 1}^{k} |\mathcal{P}_{i}| - \ell'_{|\mathcal{P}|}(\mathbb{Z}_k).\label{prfun2}
\end{eqnarray}

Further, we have
\begin{equation}\label{prfun3}
	\sum_{i = 1}^{k} |\mathcal{P}_{i}| = \sum_{i = 1}^{k} \sum_{X \in \mathcal{P}_i} 1 = \sum_{X \in \mathcal{P}} \sum_{\substack{i = 1 \\X \in \mathcal{P}_i}}^{k} 1 \ge \sum_{X \in \mathcal{P}} h = h|\mathcal{P}|.
\end{equation}
Finally, \eqref{prfun2'} and \eqref{prfun3} imply  \eqref{konfeiobfiueguye1mregdir}; \eqref{prfun2} and \eqref{prfun3} imply  \eqref{konfeiobfiueguye2mregdir}.

\medskip

Now let us prove the \textbf{sufficiency}. 
Suppose that~\eqref{hValpha},~\eqref{konfeiobfiueguye1mregdir} and~\eqref{konfeiobfiueguye2mregdir} hold. 
Note that, by \eqref{konfeiobfiueguye2mregdir} applied for  ${\cal P}=\{V\},$ we obtain that $k\ge h.$
Let $T = V$, $S = \{s_1, \dots, s_k\}$ and $p(Y) = h - d^-_\mathcal{A}(Y)$ if $\emptyset \neq Y \subseteq V$ and $0$ if  $Y = \emptyset$. Then $p$ is an intersecting supermodular set function on $T.$
Let $f(v)=0$ and $g(v)=h$ for every $v\in T$ and $f(s_i) = \ell(i), g(s_i)=\ell'(i)$ for every $i \in  \mathbb{Z}_k$. 
Note that, by \eqref{totnecessary} and \eqref{indivnecessary}, we have $\alpha \le \beta$ and $f \le g$.

\begin{cl}\label{equiv}
For all $X\subseteq S, Y\subseteq  T$ and subpartition $\mathcal{P}$ of $T-Y,$ \eqref{41}--\eqref{51} hold.
\end{cl}

\begin{proof} Since $f(v)=0$  for every $v\in T$, the claim is equivalent to the following where $Z=T-Y.$
For all $X\subseteq  \mathbb{Z}_k, Z\subseteq  V$ and subpartition $\mathcal{P}$ of $Z,$ we have
\begin{alignat}{3}
h|\mathcal{P}| - e_{\mathcal{A}}(\mathcal{P}) - |X||\mathcal{P}| &\le \ell'(\overline X), \label{41'}\\
\ell(X) + h|\mathcal{P}| - e_\mathcal{A}(\mathcal{P}) - |X||\mathcal{P}| &\le h|Z| + |X||\overline Z|, \label{42'}\\
\alpha + h|\mathcal{P}| - e_{\mathcal{A}}(\mathcal{P}) - |X||\mathcal{P}| &\le \ell'(\overline X) + h|Z| + |X||\overline Z|, \label{50'}\\
\ell(X) + h|\mathcal{P}| - e_\mathcal{A}(\mathcal{P}) - |X||\mathcal{P}| &\le \beta.\label{51'}
\end{alignat}
Now, \eqref{41'} is equivalent to \eqref{konfeiobfiueguye2mregdir} and \eqref{51'} is equivalent to \eqref{konfeiobfiueguye1mregdir}.

\medskip

If $|X|\ge h$ then, by \eqref{hValpha}, $\ell' \ge 0$ and \eqref{totnecessary}, we have 
$h|\mathcal{P}|-e_\mathcal{A}(\mathcal{P})-|X||\mathcal{P}| \le 0 \le h|V| - \alpha \le h|Z| + |X||\overline Z| - \alpha \le h|Z| + |X||\overline Z| - \alpha + \min\{\ell'(\overline X), \alpha - \ell(X)\}$,
so \eqref{42'} and \eqref{50'} follow.
\medskip

If $|X|\le h$ then, by $|Z|\ge|\mathcal{P}|,$ \eqref{totnecessary} and \eqref{indivnecessary}, we have 
$-e_\mathcal{A}(\mathcal{P})\le 0\le (h-|X|)(|Z|-|\mathcal{P}|)+(|X||V|-\ell'(X))+\min\{\ell'( \mathbb{Z}_k)-\alpha,\ell'(X)-\ell(X)\},$
 so \eqref{42'} and \eqref{50'} follow.
\end{proof}

By Claim \ref{equiv}, Theorem \ref{BFBG} can be applied and hence there exists a simple bipartite graph $\bm{G}=(S,T,E)$ such that $|\Gamma_E(Y)|\ge p(Y)$ for every $Y\subseteq V,$ $d_E(v)\le h$ for every $v\in V$, $\ell(i)\le d_E(s_i)\le\ell'(i)$ for every $i\in  \mathbb{Z}_k$ and $\alpha\le |E|\le \beta$.

Let $\bm{S_i}$ be the set of neighbors of $s_i$ in $G$ and $\bm{{\cal S}}=\{S_1,\dots,S_k\}$. Since $|{\cal S}_Y|=|\Gamma_E(Y)|\ge h - d^-_\mathcal{A}(Y)$  for every $\emptyset\neq Y\subseteq V$ and $|{\cal S}_v|=d_E(v) \le h$ for every $v\in V$, we may apply Theorem \ref{blieioeibiuzg} to $\mathcal{D}$.
Therefore there exists an $h$-regular packing of  $k$ $S_i$-hyperbranchings $B_i$ in $\mathcal{D}$.

Furthermore, since $|S_i|=d_E(s_i)$  for every $i\in  \mathbb{Z}_k$ and $\sum_{i=1}^k|S_i|=|E|$, we have that $\ell(i)\le |S_i|\le\ell'(i)$ for every $i\in  \mathbb{Z}_k$ and $\alpha \le \sum_{i=1}^k|S_i| \le \beta$. Therefore, this is an $h$-regular $(\ell, \ell')$-bordered $(\alpha, \beta)$-limited packing of $k$ hyperbranchings in $\mathcal{D}$ which concludes the proof.
\end{proof}

\subsection{Proof of Theorem \ref{kjjhvhgu}}\label{pojfpohjfop}

\begin{proof}
  In this proof we use matroids so we consider  spanning forests as  edge sets.   
  Let {\boldmath ${\sf M}_G$} $=(E,r_G)$ be the graphic matroid of $G=(V,E)$, that is the independent sets of ${\sf M}_G$ are the spanning forests of $G$ and we have 
$r_G(F)= |V| - c(V, F) \text{ for every } F\subseteq E.$ 
For $i = 1, \dots, k$, by \eqref{bjehvde}, we have $|V| - \ell(i)\ge 0$. We can hence define {\boldmath${\sf M}_i$} $= (E, r_i)$ to be the truncated matroid of ${\sf M}_G$ at $|V| - \ell(i)$, that is the independent sets of ${\sf M}_i$ are the spanning forests in $G$ with at most $|V| - \ell(i)$ edges (so at least $\ell(i)$ connected components) and, by the definition of $r_G$, we have 
\begin{eqnarray}\label{jkbfebiu2}
  r_i(F) = \min\{ |V| - c(V, F), |V| - \ell(i)\} = |V| - \max\{c(V, F), \ell(i)\} \text{ for every } F \subseteq E.  
\end{eqnarray}
  Let {\boldmath${\sf M}^*$} $=(E,r^*)$ be the sum of the matroids ${\sf M}_1,\dots, {\sf M}_k$.   The next claim implies  Theorem \ref{kjjhvhgu}.

\begin{cl}\label{ofkbebf}
The  following statements are equivalent.
    \begin{enumerate}[label=(\alph*)]
        \item There exists a packing of $k$ spanning forests in $G$ with $\ell(1), \ldots, \ell(k)$ connected components.\label{item1}
        \item There exists  an independent set $F$ in ${\sf M}^*$ of size $k|V|-\ell(\mathbb{Z}_k).$\label{item2}
        \item $e_E(\mathcal{P}) \ge k|\mathcal{P}| - \ell_{|\mathcal{P}|}(\mathbb{Z}_k)$ for every partition $\mathcal{P}$ of $V$.\label{item4}
    \end{enumerate}
\end{cl}

\begin{proof}
    {\ref{item1} $\Leftrightarrow$~\ref{item2}}: First suppose  that there exists a packing  of spanning forests $F_1, \ldots, F_k $ in $G$ such that $c(V, F_i) = \ell(i)$ for every $i\in \mathbb{Z}_k.$ Then $F_i$ is  independent  in ${\sf M}_i$   for every $i\in \mathbb{Z}_k$ and hence $F = \bigcup_{i=1}^k F_i$ is  independent  in ${\sf M}^*$. Further, we have $$|F| = \sum_{i = 1}^k |F_i| = \sum_{i = 1}^k (|V| - c(V, F_i))= \sum_{i = 1}^k (|V| - \ell(i))=k|V|-\ell(\mathbb{Z}_k).$$
\smallskip
    
Now suppose  that there exists an independent set $F$ in ${\sf M}^*$ of size $k|V|-\ell(\mathbb{Z}_k)$. Then there exist  pairwise disjoint edge sets $F_1, \ldots, F_k$ such that $\bigcup_{i = 1}^k F_i = F$ and $F_i$ is independent in ${\sf M}_i$ that is $r_i(F_i) = |F_i|$ for every 
$i\in \mathbb{Z}_k.$ It follows, by \eqref{jkbfebiu2}, that  $$k|V|-\ell(\mathbb{Z}_k)=|F| = \sum_{i = 1}^{k}|F_i|= \sum_{i = 1}^{k}r_i(F_i) \le\sum_{i = 1}^{k} (|V| - \ell(i))=k|V|-\ell(\mathbb{Z}_k),$$ so   $r_i(F_i) = |V| - \ell(i)$ for every $i\in \mathbb{Z}_k.$ Hence $F_i$ is  a spanning forest in $G$ with $|V|-|F_i|=\ell(i)$ connected components  for every $i\in \mathbb{Z}_k.$
\medskip

\noindent {\ref{item2} $\Leftrightarrow$~\ref{item4}}: First, \ref{item2} is equivalent to $r^*(E) \ge k|V|-\ell(\mathbb{Z}_k)$. This,  by Theorem \ref{partitionthm}, is equivalent to $|E - F| + \sum_{i = 1}^{k} r_i(F) \ge \sum_{i = 1}^{k} (|V| - \ell(i))$ {for every } $F \subseteq E,$ which, by \eqref{jkbfebiu2}, is equivalent to 
    \begin{eqnarray}
       |E - F|  &\ge &\sum_{i = 1}^{k} \max\{0, c(V, F) - \ell(i)\} \hskip .5truecm\text{for every } F \subseteq E.\hskip 1.7truecm\label{better3}
    \end{eqnarray}
    Secondly, since $-\ell_{|\mathcal{P}|}(i)=\max\{-\ell(i),-|\mathcal{P}|\}$ for $i = 1, \dots, k$, \ref{item4} is equivalent to 
    \begin{eqnarray}
        e_E(\mathcal{P}) &\ge & \sum_{i = 1}^{k} \max\{0, |\mathcal{P}| - \ell(i)\}\hskip 1truecm\text{for every partition } \mathcal{P} \text{ of } V.\label{betterBF2}
    \end{eqnarray}
    We hence need to prove that \eqref{better3} and \eqref{betterBF2} are equivalent.
 
    First suppose  that \eqref{better3} holds. Let $\mathcal{P}$ be a partition of $V$ and $F \subseteq E$ the set of edges that do not cross $\mathcal{P}$. Then $e_E(\mathcal{P})=|E - F|$ and $c(V, F) \ge |\mathcal{P}|$, and hence,  \eqref{better3} implies \eqref{betterBF2}.
 
    Now suppose that \eqref{betterBF2} holds. Let $F \subseteq E$ and $\mathcal{P}=\mathcal{P}(F)$. Then $|E - F|\ge e_E(\mathcal{P})$ and $c(V, F)=|\mathcal{P}|$, and hence, \eqref{betterBF2}  implies \eqref{better3}.
This concludes the proof of the claim.
\end{proof}

As mentioned above Theorem \ref{kjjhvhgu} follows from Claim \ref{ofkbebf}.
\end{proof}

\subsection{Proof of Theorem \ref{kbjelciueundi}}\label{subsecproofhopp}

\begin{proof}
The \textbf{necessity} of this theorem follows from the necessity of Theorem \ref{jjjnbbbnajkhypreg}.

To prove the \textbf{sufficiency}, let us suppose that~\eqref{bjehvde},~\eqref{kjdjhzvdh} and~\eqref{jvjvkjbkhlhio} hold. We may suppose without loss of generality that 
\begin{equation}\label{ellass}
\ell(1) \ge \ldots \ge \ell(k). 
\end{equation}
By \eqref{jvjvkjbkhlhio} applied for $\mathcal{P} = \{V\}$, we get that $k \ge h$. If $k = h$, Theorem \ref{kbjelciueundi} reduces to Theorem \ref{kjjhvhgu}. We hence suppose that $k > h$. For a packing $\mathcal{F}$ of spanning forests, we denote by {\boldmath$c_{\min}(\mathcal{F})$} the minimum number of connected components of a forest in $\mathcal{F}$.

\begin{Lemma}\label{lemmafspan}
There exist an index $i$ and  a packing $\mathcal{F}=\{F_1, \ldots, F_{h}\}$ of $h$ spanning forests in $G$ such that the following hold
\begin{itemize}
 \setlength\itemsep{-0cm}
	\item[(i)] $c(F_j) = \ell(j)$ for $j\in\{1 , \dots, i\}$,
        	\item[(ii)]  $c_{\min}(\mathcal{F}) \ge \ell(i + 1)$,
        	\item[(iii)]  $\sum_{j = i+1}^{h}c(F_j)=\sum_{j = i+1}^{k} \ell(j)$.
\end{itemize}
\end{Lemma}

\begin{proof}
Let {\boldmath$\ell(0)$} $= |V|.$  Then, by~\eqref{kjdjhzvdh}, the following inequality holds for $i' = 0$.
Let hence $\bm{i'}$ be the maximum integer such that
\begin{equation}\label{uselful}
	(h - i')\ell(i') \ge \sum_{j = i' + 1}^{k}\ell(j).
\end{equation}
By $k>h,$ we have $i'<h$. Then, by the maximality of $i'$, we have $(h - i' - 1)\ell(i' + 1) < \sum_{j = i' + 2}^{k}\ell(j)$ that is
\begin{equation}\label{uselful2}
(h - i')\ell(i' + 1) < \sum_{j = i' + 1}^{k}\ell(j). 
\end{equation}
Let the function {\boldmath$\ell'$} $: \{1, \ldots, h\} \rightarrow \mathbb{Z}_{>0}$ be obtained from the function $\ell$ by evenly distributing $\ell(i' + 1), \ldots, \ell(k)$  over $\ell'(i' + 1), \ldots, \ell'(h)$ that is, for  {\boldmath$\ell^*$} $= \lceil\frac{\sum_{j = i' + 1}^{k}\ell(j)}{h - i'}\rceil,$
\begin{eqnarray}
 	\ell'(j) &=& \ell(j) \hskip 1.47truecm \text{ for } 1 \le j \le i',\label{mkjfeljjk2}\\
	\ell'(j) &\in& \{\ell^*, \ell^* - 1\}  \ \ \ \text{ for }  i' + 1 \le j \le h,\label{mkjfeljjk3}\\
	\sum_{j = i' + 1}^{k} \ell(j) &=& \sum_{j = i' + 1}^{h} \ell'(j).\label{mkjfeljjk}
\end{eqnarray}
 By \eqref{uselful} and \eqref{uselful2}, we have 
\begin{equation}\label{kjvdkjzvdjk}
\ell(i') \ge \ell^* > \ell(i' + 1).
\end{equation}

\begin{cl}\label{kjbkjbfk}
The conditions of Theorem \ref{kjjhvhgu} hold for $\ell'$, that is 
\begin{align}
	|V| &\ge \ell'(j) &&\text{for every } 1 \le j \le h, \label{newBF1}\\
	e_E(\mathcal{P}) &\ge h|\mathcal{P}| - \sum_{j = 1}^{h} \ell'_{|\mathcal{P}|}(j) &&\text{for every partition $\mathcal{P}$ of $V$}. \label{newBF2}
\end{align}
\end{cl}

\begin{proof}
We first consider \eqref{newBF1}. For  $1\le j\le i',$  we have, by \eqref{bjehvde} and \eqref{mkjfeljjk2}, that $|V| \ge \ell(j)=\ell'(j).$  For  $i' + 1\le j\le h,$ we get, by \eqref{bjehvde}, \eqref{kjvdkjzvdjk} and \eqref{mkjfeljjk3}, that  $|V| \ge \ell(i')\ge  \ell^*\ge \ell'(j),$ so \eqref{newBF1} holds. 

\noindent To show that~\eqref{newBF2}  holds, let $\mathcal{P}$ be a partition of $V$. 

If $|\mathcal{P}| \le \ell^* - 1$ then, by \eqref{kjvdkjzvdjk}, \eqref{ellass} and \eqref{mkjfeljjk2} for $1\le j\le i',$ and, by \eqref{mkjfeljjk3} for $i' + 1\le j\le h,$   we get $|\mathcal{P}| \le \ell^* - 1\le \ell'(j)$ for $1\le j\le k$, thus  $e_E(\mathcal{P}) \ge 0=h|\mathcal{P}| - \sum_{j = 1}^{h} \ell'_{|\mathcal{P}|}(j)$, so \eqref{newBF2} holds.

If $|\mathcal{P}| \ge \ell^*$ then we have, by  \eqref{kjvdkjzvdjk} and \eqref{ellass},   that $|\mathcal{P}| \ge \ell^*> \ell(j)$ that is $\ell_{|\mathcal{P}|}(j)=\ell(j)$ for $i' + 1\le j\le k$, and, by \eqref{mkjfeljjk3},  we get $|\mathcal{P}| \ge \ell^* \ge \ell'(j)$  that is $\ell'_{|\mathcal{P}|}(j)=\ell'(j)$  for $i' + 1\le j\le k.$ Thus, by \eqref{mkjfeljjk2}  and \eqref{mkjfeljjk}, we have 

\begin{align}
	\sum_{j = 1}^{k} \ell_{|\mathcal{P}|}(j) = \sum_{j = 1}^{i'} \ell_{|\mathcal{P}|}(j) + \sum_{j = i' + 1}^{k} \ell(j) 
	= \sum_{j = 1}^{i'} \ell'_{|\mathcal{P}|}(j) + \sum_{j = i' + 1}^{h} \ell'(j)
	= \sum_{j = 1}^{h} \ell'_{|\mathcal{P}|}(j).\label{filou}
\end{align}
By~\eqref{jvjvkjbkhlhio} and~\eqref{filou}, we get that \eqref{newBF2} holds:
	$e_E(\mathcal{P}) \ge h |\mathcal{P}| - \sum_{j = 1}^{k} \ell_{|\mathcal{P}|}(j)=h |\mathcal{P}| - \sum_{j = 1}^{h} \ell'_{|\mathcal{P}|}(j).$
\end{proof}

 According to Claim \ref{kjbkjbfk} and Theorem \ref{kjjhvhgu} applied for $\ell'$, there exists a packing $\mathcal{F}$ of $h$ spanning  forests in $G$ with $\ell'(1),\dots,\ell'(h)$ connected components. By \eqref{mkjfeljjk2}, \eqref{mkjfeljjk3}, \eqref{mkjfeljjk}  and $c_{\min}(\mathcal{F})\ge\ell^*-1\ge \ell(i+1),$ the index $i'$ and the packing $\mathcal{F}$ satisfy (i), (ii) and (iii), which completes the proof.
\end{proof}

By Lemma \ref{lemmafspan}, there exists   a packing  of $h$ spanning forests in $G$ satisfying (i)--(iii). 
Let {\boldmath$\mathcal{F}$} $= \mathcal{F}_1 \cup \mathcal{F}_2$ be such a packing with maximum $i=|\mathcal{F}_1|$  and with minimum $c_{\min}(\mathcal{F}_2).$ 
Let {\boldmath$F_{\min}$} be a member of  $\mathcal{F}_2$ such that $c(F_{\min})=c_{\min}(\mathcal{F}_2)$. 
By (ii), we have $c_{\min}(\mathcal{F}_2)\ge c_{\min}(\mathcal{F}) \ge \ell(i + 1).$ If $c_{\min}(\mathcal{F}_2)=\ell(i + 1)$ then by  moving $F_{\min}$ from $\mathcal{F}_2$ to $\mathcal{F}_1$ we obtain a packing satisfying (i), (ii), and (iii) that contradicts the maximality of $i$.
Therefore, we have 
\begin{equation}\label{cgfe}
c_{\min}(\mathcal{F}_2) > \ell(i + 1).
\end{equation}

Let {\boldmath$\mathcal{T}$} be the set of the connected components of the forests of $\mathcal{F}_2.$
    \begin{cl}\label{claim1}
        The vertex set of an element of $\mathcal{T}$ does not cross $ \mathcal{P}(F_{\min})$.
    \end{cl}

    \begin{proof}
        Suppose for a contradiction that there exists a connected component  of a spanning forest {\boldmath$F$} in $\mathcal{F}_2$  whose vertex set crosses $\mathcal{P}(F_{\min})$. Let {\boldmath$a$} be an edge of $F$ that connects two connected components  of $F_{\min}$, {\boldmath$\mathcal{F}'_2$} $= \mathcal{F}_2 - \{F, F_{\min}\} + \{F - a, F_{\min} + a\}$ and $\mathcal{F}' = \mathcal{F}_1 \cup \mathcal{F}'_2$. Then  $c(F-a) = c(F)+1$ and, by \eqref{cgfe}, we have $c_{\min}(\mathcal{F}'_2)=c(F_{\min} + a) = c(F_{\min}) - 1=c_{\min}(\mathcal{F}_2)-1\ge \ell(i + 1)$. 
        Then $c_{\min}(\mathcal{F}')=\min\{c_{\min}(\mathcal{F}_1), c_{\min}(\mathcal{F}'_2)\}\ge \ell(i + 1).$
        Hence $\mathcal{F}' $ is a packing of $h$ spanning forests of $G$ satisfying (i), (ii), (iii), and maximizing $i$; which contradicts the minimality of $c_{\min}(\mathcal{F})$.
    \end{proof}

Let {\boldmath$G_j$} $=(V,E_j)$ where {\boldmath$E_j$} $=E(\mathcal{F}_j)$ for $j=1,2.$ 

    \begin{cl}\label{uncrossclaim}
         There exists an $(h - i)$-regular packing of $k - i$ forests in $G_2$ with $\ell(i + 1), \ldots, \ell(k)$ connected components.
    \end{cl}

\begin{proof}
Let $\{${\boldmath$V_1$}$, \ldots,${\boldmath$V_p$}$\} = \mathcal{P}(F_{\min})$. For  $T\in\mathcal{T}$, let {\boldmath$\lambda(T)$} and {\boldmath$\mu(T)$} be the indices such that $V(T)\in\mathcal{P}(F_{\lambda(T)})$ and $V(T) \subseteq V_{\mu(T)}$. 
Note that $\mu(T)$ is well defined because, by Claim \ref{claim1}, $V(T)$ does not cross $\mathcal{P}(F_{\min})$. We introduce an order on $\mathcal{T}$. For $T, T'\in\mathcal{T}$, we say $T$ precedes $T'$ if either $\lambda(T) < \lambda(T')$ or  $\lambda(T) = \lambda(T')$ and $\mu(T) \le \mu(T');$ 
in case of equality we can arbitrarily choose the order. Note that if $T, T'\in\mathcal{T}$  intersect then $\lambda(T) \neq \lambda(T')$ and $\mu(T) = \mu(T')$, so  there are at least $p - 1$ elements of $\mathcal{T}$ between $T$ and $T'$ in the order. For $j \in \{i + 1, \ldots, k\}$, let {\boldmath$F^*_j$}  be  the union of the first $\ell(j)$  consecutive elements of $\mathcal{T}$ in the order that have not been taken yet. By \eqref{ellass} and \eqref{cgfe}, we have $\ell(j)\le\ell(i + 1) < c(F_{\min})=|\mathcal{P}(F_{\min})|=p.$ Hence each $F^*_j$ contains  pairwise disjoint elements of $\mathcal{T}$, implying that each $F^*_j$ is a forest. By (iii), each element of $\mathcal{T}$ is used exactly once in  {\boldmath$\mathcal{F}^*$} $=\{F^*_{i+1}, \ldots, F^*_k\}$. Since  $\mathcal{F}_2$ is an $(h-i)$-regular  packing,  $\mathcal{F}^*$ is an $(h - i)$-regular packing of $k - i$ forests in $G_2$ with $\ell(i + 1), \ldots, \ell(k)$ connected components.
\end{proof}
    
By Claim \ref{uncrossclaim}, there exists an $(h - i)$-regular packing $\mathcal{F}^*$ of $k - i$ forests in $G_2$ with 
$\ell(i + 1), \ldots, \ell(k)$ connected components. Recall that $\mathcal{F}_1$ is an $i$-regular packing of $i$ forests in $G_1$ with 
$\ell(1), \ldots, \ell(i)$ connected components. Further,  $E_1 \cap E_2=\emptyset$ and $E(\mathcal{F}^*)\subseteq E_2$. Therefore, $\mathcal{F}_1\cup\mathcal{F}^*$ is an $h$-regular packing of $k$ forests in $G$ with
$\ell(1), \ldots, \ell(k)$ connected components that completes the proof of Theorem \ref{kbjelciueundi}.
\end{proof}

\subsection{Proof of Theorem \ref{livzdljzvdmz}}\label{subsecproofregforwithbounds}

\begin{proof}
The \textbf{necessity} of this theorem follows from the necessity of Theorem \ref{jjjnbbbnajkhypreg}.

\medskip

Let us now prove the \textbf{sufficiency}. Suppose that~\eqref{hValpha},~\eqref{konfeiobfiueguye1m} and~\eqref{konfeiobfiueguye2m} hold.
We distinguish the case when $h|V| < \beta$. According to~\eqref{totnecessary} and~\eqref{indivnecessary}, there exists $\ell^* : \mathbb{Z}_k \rightarrow \mathbb{Z}_{\ge 0}$ such that $\ell'(i) \ge \ell^*(i) \ge \ell(i)$ for every $1 \le i \le k$ and $\ell^*(\mathbb{Z}_k) = h|V|$. By applying the proof of Claim~\ref{uncrossclaim} to a packing of $h$ spanning forests of isolated vertices, we obtain an $h$-regular packing of $k$ forests with $\ell^*(1), \ldots, \ell^*(k)$ connected components. It is $(\ell, \ell')$-bordered by construction and by~\eqref{hValpha} and $h|V| < \beta$ it is also $(\alpha, \beta)$-limited. 

\medskip

Suppose now that $h|V| \ge \beta$. Let {\boldmath$\ell^*$} $: \mathbb{Z}_k \rightarrow \mathbb Z_{\ge 0}$ be a function that satisfies
\begin{align}
    \beta &\ge \ell^*(\mathbb{Z}_k), \label{*1}\\
    \ell'(i) &\ge \ell^*(i) \ge \ell(i) &&\text{for every } 1 \le i \le k, \label{*2}\\
   \beta - \ell^*(\mathbb{Z}_k)+\ell^*_{|\mathcal{P}|}(\mathbb{Z}_k) &\ge h|\mathcal{P}|- e_E(\mathcal{P})     &&\text{for every partition } \mathcal{P} \text{ of } V, \label{cond3*}\\
    \ell^*(\mathbb{Z}_k) & \text{ is maximum}.
\end{align}

Note that such $\ell^*$ exists because, by~\eqref{totnecessary},~\eqref{indivnecessary} and~\eqref{konfeiobfiueguye1m}, $\ell$ satisfies~\eqref{*1},~\eqref{*2} and~\eqref{cond3*}.
   
\begin{Lemma}\label{claimeq}
    $\ell^*(\mathbb{Z}_k) = \beta$.
\end{Lemma}
   
\begin{proof}
Suppose for a contradiction that $\ell^*(\mathbb{Z}_k) < \beta$. 

\begin{cl}\label{obdouyzcdytzcd}
	There exists $p^*\in\{0,\dots,|V| - 1\}$ and $j\in \mathbb{Z}_k$ such that 
\begin{align}
    \beta - \ell^*(\mathbb{Z}_k) + \ell^*_{p}(\mathbb{Z}_k) &> \ell'_{p}(\mathbb{Z}_k) \hskip 1truecm\text{for every } 0 \le p \le p^*, \label{f*gtf}\\
	\ell^*(j) \le p^* &< \ell'(j). \label{fdfgtf}
\end{align}
\end{cl}

\begin{proof}
Note that, by $\ell^*(\mathbb{Z}_k) < \beta$, we have $\beta - \ell^*(\mathbb{Z}_k) + \ell^*_{0}(\mathbb{Z}_k) = \beta - \ell^*(\mathbb{Z}_k) > 0 = \ell'_{0}(\mathbb{Z}_k)$.

Further, we have $\beta - \ell^*(\mathbb{Z}_k) + \ell^*_{|V|}(\mathbb{Z}_k) = \beta \le \ell'(\mathbb{Z}_k) = \ell'_{|V|}(\mathbb{Z}_k)$.

It follows that there exists a maximum $p^*\in\{0, \ldots, |V| - 1\}$ satisfying \eqref{f*gtf}.
By the maximality of $p^*$ and \eqref{f*gtf}, we have $\beta - \ell^*(\mathbb{Z}_k) + \ell^*_{p^* + 1}(\mathbb{Z}_k) \le \ell'_{p^* + 1}(\mathbb{Z}_k)$ and $\beta - \ell^*(\mathbb{Z}_k) + \ell^*_{p^*}(\mathbb{Z}_k) > \ell'_{p^*}(\mathbb{Z}_k)$. This implies that there exists $j \in \mathbb{Z}_k$ such that $\min\{p^* + 1, \ell'(j)\} = \min\{p^*, \ell'(j)\} + 1$ and $\min\{p^* + 1, \ell^*(j)\} = \min\{p^*, \ell^*(j)\}$. Hence $\ell'(j) \ge p^* + 1$ and $\ell^*(j) \le p^*$, that is \eqref{fdfgtf} holds.
\end{proof}

By Claim \ref{obdouyzcdytzcd}, there exists {\boldmath$p^*\in\mathbb{Z}$} and {\boldmath$j$} $\in \mathbb{Z}_k$ satisfying \eqref{f*gtf} and \eqref{fdfgtf}.
Let {\boldmath$\ell^{+}$} $: \mathbb{Z}_k \rightarrow \mathbb Z_{>0}$ be defined as follows: $\ell^{+}(i) = \ell^*(i)$ for all $i \in \mathbb{Z}_k - \{j\}$ and $\ell^{+}(j) = \ell^*(j) + 1$. 

\begin{cl}\label{bbeibfbf}
$\ell^+$ satisfies~\eqref{*1},~\eqref{*2} and~\eqref{cond3*}.
\end{cl}

\begin{proof}
By $\ell^*(\mathbb{Z}_k) < \beta$, we have $\ell^{+}(\mathbb{Z}_k) = \ell^*(\mathbb{Z}_k) + 1 \le \beta$, so \eqref{*1} holds for $\ell^+$.

By \eqref{*2}, we have $\ell(i) \le  \ell^*(i)=\ell^+(i) = \ell^*(i) \le \ell'(i)$ for every $i \in \mathbb{Z}_k - \{j\}$ and, by \eqref{fdfgtf},  we have  $\ell(j) \le \ell^*(j)\le \ell^+(j) = \ell^*(j) + 1 \le \ell'(j)$, so \eqref{*2} holds for $\ell^+$. 

Let $\mathcal{P}$ be a partition of $V.$ If $|\mathcal{P}| > p^*$ then, by $\ell^*(j) \le p^*$ and~\eqref{cond3*}, we have
$
	\beta - \ell^+(\mathbb{Z}_k) + \ell^+_{|\mathcal{P}|}(\mathbb{Z}_k)
	= \beta - (\ell^*(\mathbb{Z}_k)+1) + (\ell^*_{|\mathcal{P}|}(\mathbb{Z}_k) + 1) 
	\ge h|\mathcal{P}| - e_E(\mathcal{P}).
$
If $|\mathcal{P}| \le p^*$ then, by~\eqref{f*gtf} and~\eqref{konfeiobfiueguye2m}, we have
$	
	\beta - \ell^+(\mathbb{Z}_k) + \ell^+_{|\mathcal{P}|}(\mathbb{Z}_k) = \beta - (\ell^*(\mathbb{Z}_k) + 1) + \ell^*_{|\mathcal{P}|}(\mathbb{Z}_k)
	\ge \ell'_{|\mathcal{P}|}(\mathbb{Z}_k) 
	\ge h|\mathcal{P}| - e_E(\mathcal{P}).
$
Therefore, \eqref{cond3*} holds for $\ell^+$, that completes the proof of Claim \ref{bbeibfbf}.
\end{proof}

Claim \ref{bbeibfbf} and $\ell^+(\mathbb{Z}_k)>\ell^*(\mathbb{Z}_k)$ contradict the maximality of $\ell^*(\mathbb{Z}_k)$, completing the proof of Lemma \ref{claimeq}.
\end{proof}

By Lemma \ref{claimeq} and \eqref{cond3*}, \eqref{jvjvkjbkhlhio} holds for $\ell^*$.
By~\eqref{*2}, \eqref{bjehvde}  holds for $\ell^*$. Finally, by Lemma \ref{claimeq}, we have $h|V| \ge \beta = \ell^*(\mathbb{Z}_k)$, so \eqref{kjdjhzvdh} holds for $\ell^*$. Hence, by Theorem~\ref{kbjelciueundi},   there exists an $h$-regular packing of $k$ forests with respectively $\ell^*(1), \ldots, \ell^*(k)$ connected components. By \eqref{*2}, Lemma \ref{claimeq} and \eqref{totnecessary}, this packing is $(\ell, \ell')$-bordered and $(\alpha, \beta)$-limited, which completes the proof of Theorem~\ref{livzdljzvdmz}.
\end{proof}

\subsection{Proof of Theorem \ref{jjjnbbbnajkhypreg}}\label{hdfoibdoibio}

\begin{proof}
We first show the \textbf{necessity}. Suppose that there exists an $h$-regular $(\ell, \ell')$-bordered $(\alpha, \beta)$-limited packing of $k$ rooted hyperforests in $\mathcal{G}$. Then we can orient $\mathcal{G}$ to get a dypergraph $\mathcal{D}$ that has an $h$-regular $(\ell, \ell')$-bordered $(\alpha, \beta)$-limited packing of $k$ hyperbranchings. Then the necessity of Theorem \ref{jjjnbbbnajkhypregdir} implies the necessity of Theorem \ref{jjjnbbbnajkhypreg}.
\medskip

To show the \textbf{sufficiency}, suppose that~\eqref{hValpha},~\eqref{konfeiobfiueguye1mreg} and~\eqref{konfeiobfiueguye2mreg} hold in $\mathcal{G}$. 

\begin{Lemma}\label{trimminglemma}
The hypergraph $\mathcal{G}$ can  be trimmed to a graph $G$ that satisfies \eqref{hValpha}, \eqref{konfeiobfiueguye1m} and \eqref{konfeiobfiueguye2m}.
\end{Lemma}

\begin{proof}
We prove the lemma by induction on $\sum_{X \in \mathcal{E}} |X|.$
If for every  $X \in \mathcal{E}$, $|X| = 2$ then $\mathcal{G}$ is a graph and, \eqref{konfeiobfiueguye1mreg} and \eqref{konfeiobfiueguye2mreg} coincide with \eqref{konfeiobfiueguye1m} and \eqref{konfeiobfiueguye2m}.
Otherwise, there exists $X \in \mathcal{E}$ such that $|X| \ge 3$. We show that we can always remove a vertex from $X$ without violating \eqref{hValpha}, \eqref{konfeiobfiueguye1mreg} or \eqref{konfeiobfiueguye2mreg}. Note that the removal of a vertex from a hyperedge does not effect \eqref{hValpha}.
Suppose that no vertex of $X$ can be removed from $X$ without violating the conditions, that is, for every $x \in X$, at least one of \eqref{konfeiobfiueguye1mreg} and \eqref{konfeiobfiueguye2mreg} is violated after the removal of $x$ from $X$. By $|X| \ge 3$, there must be at least two vertices violating the same condition. Let $a, b \in X$ such that their removals violate the same condition. Since this condition is satisfied before the removal of the vertex, there exist partitions $\mathcal{P}_a$ and $\mathcal{P}_b$ of $V$, such that
either \eqref{tighti} or \eqref{tightii} hold, $e_{\mathcal{E}}(\mathcal{P}_a)$ decreases when removing $a$ from $X$ and $e_{\mathcal{E}}(\mathcal{P}_b)$ decreases when removing $b$ from $X$.
\begin{alignat}{4}
	&&\ell_{|\mathcal{P}_a|}(\mathbb{Z}_k) + e_\mathcal{E}(\mathcal{P}_a) - h|\mathcal{P}_a| &= \ell(\mathbb{Z}_k) &&- \beta &= \ell_{|\mathcal{P}_b|}(\mathbb{Z}_k) + e_\mathcal{E}(\mathcal{P}_b) - h|\mathcal{P}_b|,\label{tighti}\\
	&&\ell'_{|\mathcal{P}_a|}(\mathbb{Z}_k)+e_\mathcal{E}(\mathcal{P}_a) - h|\mathcal{P}_a| &= &&0 &= \ell'_{|\mathcal{P}_b|}(\mathbb{Z}_k)+e_\mathcal{E}(\mathcal{P}_b) - h|\mathcal{P}_b|. \label{tightii}
\end{alignat}
As removing $a$ from $X$ decreases $e_{\mathcal{E}}(\mathcal{P}_a)$, there exist $X_a, X_{\overline{a}} \in \mathcal{P}_a$ such that $a \in X_a$ and $X - a \subseteq X_{\overline{a}}$.  The same is also true for $b$ and $\mathcal{P}_b$, i.e. there exist $X_b, X_{\overline{b}} \in \mathcal{P}_b$ such that $b \in X_b$, and $X - b \subseteq X_{\overline{b}}$.
Let {\boldmath$\mathcal{P}_\sqcup$} $= \mathcal{P}_a \sqcup \mathcal{P}_b$ and {\boldmath$\mathcal{P}_\sqcap$} $ = \mathcal{P}_a \sqcap \mathcal{P}_b$.
Since $X_{\overline{a}} \cap X_{\overline{b}} \supseteq X - a - b \neq \emptyset$, $X_a \cap X_{\overline{b}} \supseteq  \{a\} \neq \emptyset$ and $X_{\overline{a}} \cap X_b \supseteq  \{b\} \neq \emptyset$, we get that $X \subseteq X_a \cup X_b \cup X_{\overline{a}} \cup X_{\overline{b}}$ is contained in a member of $\mathcal{P}_\sqcup$. Thus $X$ does not cross $\mathcal{P}_\sqcup$, i.e. 
$$e_{\{X\}}(\mathcal{P}_a) + e_{\{X\}}(\mathcal{P}_b) > e_{\{X\}}(\mathcal{P}_\sqcap) + e_{\{X\}}(\mathcal{P}_\sqcup).$$ Then, according to Lemma \ref{submodeep},  we obtain that
\begin{equation}\label{strictuncross}
	e_{\mathcal{E}}(\mathcal{P}_a) + e_{\mathcal{E}}(\mathcal{P}_b) > e_{\mathcal{E}}(\mathcal{P}_\sqcap) + e_{\mathcal{E}}(\mathcal{P}_\sqcup).
\end{equation}
By \eqref{kvkjvk3} and \eqref{usecl1}, we can apply Claim \ref{submodpart} for $\ell^* \in \{\ell, \ell'\}$ and every $i \in \mathbb{Z}_k,$ and we get
$$
	\min\{\ell^*(i), |\mathcal{P}_a|\} + \min\{\ell^*(i), |\mathcal{P}_b|\} \ge \min\{\ell^*(i), |\mathcal{P}_\sqcap|\} + \min\{\ell^*(i), |\mathcal{P}_\sqcup|\}.
$$
By summing up these inequalities for $i \in \mathbb{Z}_k,$ it follows that
\begin{equation}\label{lsubmodular}
	\ell_{|\mathcal{P}_a|}^*(\mathbb{Z}_k) + \ell_{|\mathcal{P}_b|}^*(\mathbb{Z}_k) \ge \ell_{|\mathcal{P}_\sqcap|}^*(\mathbb{Z}_k) + \ell_{|\mathcal{P}_\sqcup|}^*(\mathbb{Z}_k).
\end{equation}
If the violated condition was \eqref{konfeiobfiueguye1mreg} then, by \eqref{tighti}, \eqref{lsubmodular} for $\ell$, \eqref{strictuncross}, \eqref{konfeiobfiueguye1mreg}  and \eqref{kvkjvk3}, we have
\begin{align*}
	h(|\mathcal{P}_a| + |\mathcal{P}_b|) &= 2(\beta - \ell(\mathbb{Z}_k)) + \ell_{|\mathcal{P}_a|}(\mathbb{Z}_k) + \ell_{|\mathcal{P}_b|}(\mathbb{Z}_k) + e_{\mathcal{E}}(\mathcal{P}_a) + e_{\mathcal{E}}(\mathcal{P}_b) \\
	&> 2(\beta - \ell(\mathbb{Z}_k)) + \ell_{|\mathcal{P}_\sqcap|}(\mathbb{Z}_k) + \ell_{|\mathcal{P}_\sqcup|}(\mathbb{Z}_k) + e_{\mathcal{E}}(\mathcal{P}_\sqcap) + e_{\mathcal{E}}(\mathcal{P}_\sqcup) \\
	&\ge h(|\mathcal{P}_\sqcap| + |\mathcal{P}_\sqcup|) \\
	&= h(|\mathcal{P}_a| + |\mathcal{P}_b|),
\end{align*}
 a contradiction.

\noindent If the violated condition was \eqref{konfeiobfiueguye2mreg} then, by \eqref{tightii},  \eqref{lsubmodular} for $\ell'$, \eqref{strictuncross}, \eqref{konfeiobfiueguye2mreg}, and \eqref{kvkjvk3},  we have
\begin{align*}
	h(|\mathcal{P}_a| + |\mathcal{P}_b|) &= \ell_{|\mathcal{P}_a|}'(\mathbb{Z}_k) + \ell_{|\mathcal{P}_b|}'(\mathbb{Z}_k) + e_{\mathcal{E}}(\mathcal{P}_a) + e_{\mathcal{E}}(\mathcal{P}_b) \\
	&> \ell_{|\mathcal{P}_\sqcap|}'(\mathbb{Z}_k) + \ell_{|\mathcal{P}_\sqcup|}'(\mathbb{Z}_k) + e_{\mathcal{E}}(\mathcal{P}_\sqcap) + e_{\mathcal{E}}(\mathcal{P}_\sqcup) \\
	&\ge h(|\mathcal{P}_\sqcap| + |\mathcal{P}_\sqcup|) \\
	&= h(|\mathcal{P}_a| + |\mathcal{P}_b|),
\end{align*}
 a contradiction. 
 
 The proof of Lemma \ref{trimminglemma} is complete.
\end{proof}

By Lemma \ref{trimminglemma}, the hypergraph $\mathcal{G}$ can be trimmed to a graph $G$ satisfying \eqref{hValpha}, \eqref{konfeiobfiueguye1m} and \eqref{konfeiobfiueguye2m}. Therefore, according to Theorem \ref{livzdljzvdmz}, there exists an $h$-regular $(\ell, \ell')$-bordered $(\alpha, \beta)$-limited packing of $k$ $S_i$-forests in $G$. We can orient every $S_i$-forest in the packing to get an $S_i$-branching, which can be obtained by trimming from an $S_i$-hyperbranching in an orientation of an $S_i$-hyperforest of $\mathcal{G}$. This way we obtained the required packing which completes the proof of the theorem.
\end{proof}



\begin{thebibliography}{99}
\bibitem{BF2} {K. B\'erczi, A. Frank,} \textit{Variations for Lov\'asz' submodular ideas,} in Building Bridges, Springer, (2008) 137--164.
\bibitem{BF} {K. B\'erczi, A. Frank,} \textit{Packing arborescences,}  in: S. Iwata (Ed.), RIMS Kokyuroku Bessatsu B23: Combinatorial Optimization and Discrete Algorithms, Lecture Notes, (2010) 1--31.
\bibitem{BF3} {K. B\'erczi, A. Frank,} \textit{Supermodularity in Unweighted Graph Optimization I: Branchings and Matchings,}  Math. Oper. Res. 43(3) (2018) 726--753.
\bibitem{cai} M. C. Cai, \textit{Common root function of two graphs}, J. Graph Theory, 13 (1989) 249--256.
\bibitem{Sz}  {O. Durand de Gevigney, V. H. Nguyen, Z. Szigeti}, \textit{Matroid-Based Packing of Arborescences}, SIAM J. Discret. Math. 27(1) (2013)  567--574.
\bibitem{Egy} {J. Edmonds}, \textit{Edge-disjoint branchings}, in {Combinatorial Algorithms}, B. Rustin ed., Academic Press, New York, (1973)  91--96.
\bibitem{EF} J. Edmonds and D. R. Fulkerson, \textit{Transversals and matroid partition}, J. Res. Nat. Bur. Standards, B69 (1965) 147--53.
\bibitem{FKLSzT}  {Q. Fortier, Cs. Kir\'aly, M. L\'eonard, Z. Szigeti, A. Talon}, \textit{Old and new results on packing arborescences}, Discret. Appl. Math. 242 (2018) 26--33. 
\bibitem{FKST} {Q. Fortier, Cs. Kir\'aly, Z. Szigeti, S. Tanigawa}, \textit{On packing spanning arborescences with matroid constraint},  J. Graph Theory 93(2) (2020) 230--252.
\bibitem{FA78} A.~Frank, \textit{On disjoint trees and arborescences,} In  Algebraic Methods in Graph Theory, 25, Colloquia Mathematica Soc. J. Bolyai, North-Holland, (1978) 59--169.
\bibitem{F2} {A. Frank}, \textit{Rooted $k$-connections in digraphs}, Discret. Appl. Math. 157 (2009) 1242--1254.
\bibitem{book} {A. Frank}, Connections in Combinatorial Optimization, Oxford University Press, 2011.
\bibitem {fkiki} A. Frank, T. Kir\'aly,  Z. Kir\'aly, \textit{On the orientation of graphs and hypergraphs,} Discret. Appl. Math. 131(2) (2003) 385--400.
\bibitem {fkk} A. Frank, T. Kir\'aly, M. Kriesell, \textit{On decomposing a hypergraph into $k$ connected sub-hypergraphs,} Discret. Appl. Math. 131 (2) (2003) 373--383.
\bibitem {ft} A. Frank, \'E. Tardos,  \textit{An application of submodular flows}, Linear Algebra and its Applications, 114/115, (1989) 329--348.
\bibitem{gy} H. Gao, D. Yang, \textit{Packing of maximal independent mixed arborescences}, Discret. Appl. Math. 289 (2021) 313--319.
\bibitem{gy2} H. Gao, D. Yang, \textit{Packing of spanning mixed arborescences}, J.  Graph Theory, Volume 98, Issue 2 (2021) 367--377.
\bibitem{GJ} M. R. Garey, D. S. Johnson, \textit{Computers and In-tractability. A Guide to the Theory of NP-Completeness}, W. H. Freeman, 1979.
\bibitem{hopp} P. Hoppenot, \textit{Packing forests,} \ Master Thesis, UGA-Grenoble INP, (2023).
\bibitem{HSz4} {F. H\"orsch, Z. Szigeti,} \textit{Reachability in arborescence packings}, Discret. Appl. Math. 320 (2022) 170--183.
\bibitem{HSz5} {F. H\"orsch, Z. Szigeti,} \textit{Packing of mixed hyperarborescences with flexible roots via matroid intersection}, Electronic Journal of Combinatorics 28 (3) (2021) P3.29.
\bibitem{japan}  {N. Kamiyama, N. Katoh, A. Takizawa}, \textit{Arc-disjoint in-trees in directed graphs}, Comb. 29 (2009) 197--214.
\bibitem{KT} { N. Katoh, S. Tanigawa,} \textit{Rooted-tree decomposition with matroid constrains and the infinitesimal rigidity of frameworks with boundaries}, SIAM J. Discret. Math. 27(1), (2013) 155--185. 
\bibitem {cskir} {Cs. Kir\'aly}, \textit{On maximal independent arborescence packing}, SIAM J. Discret. Math. 30(4) (2016) 2107--2114.
\bibitem{KSzT}{Cs. Kir\'aly, Z. Szigeti, S. Tanigawa,} \textit{Packing of arborescences with matroid constraints via matroid intersection}, {Math. Program.} {181} (2020) 85--117.
\bibitem{kirkpat}  D. G. Kirkpatrick, P. Hell, \textit{On the complexity of general graph factor problems}, SIAM J. Comput. Vol. 12 No. 3 (1983) 601 -- 609.
\bibitem{Lov} {L. Lov\'asz}, \textit{On two minimax theorems in graphs}, J. Comb. Theory, Ser. B 21 (1976) 96--103. 
\bibitem{MMSz} M. Martin, \textit{Packing arborescences}, Master Thesis, UGA-Grenoble INP, (2022).
\bibitem{MSz}{T. Matsuoka, Z. Szigeti,}  \textit{Polymatroid-Based Capacitated Packing of Branchings,}  Discret. Appl. Math. {270} (2019) 190--203.
\bibitem{NW} C. St. J. A. Nash-Williams, \textit{Edge-disjoints spanning trees of finite graphs}, Journal of the London Mathematical Society, 36 (1961) 445--450.
\bibitem{PCK}{Y. H. Peng, C. C. Chen, K. M. Koh,} \textit{On edge-disjoint spanning forests of multigraphs}, Soochow Journal of Mathematics 17 (1991) 317--326.
\bibitem{schaefer} T. J. Schaefer, The Complexity of Satisfiability Problems,  {\em Proceedings of the Tenth Annual ACM Symposium on Theory of Computing, STOC 78.} 3 (1978), 216--226.
\bibitem{szighyp}{Z. Szigeti,} \textit{Packing mixed hyperarborescences}, Discrete Optimization 50 (2023) 100811.
\bibitem{Tu} W.T. Tutte, \textit{On the problem of decomposing a graph into $n$ connected factors}, Journal of the London Mathematical Society, 36 (1961) 221--230.
\end{thebibliography}
\end{document}